\RequirePackage{fix-cm}
\documentclass[smallextended,reqno]{svjour3}       
\smartqed  
\usepackage{graphicx}
\usepackage{amsmath}
\usepackage{amsfonts}
\usepackage{mathrsfs}
\usepackage{amssymb}
\usepackage[english]{babel}
\usepackage{a4wide}
\usepackage{pgf}
\usepackage{subfigure}
\newtheorem{assumption}{Assumption}[section]

\newcommand{\essinf}{\mathop{\mathrm{ess~inf}}}

\newcommand{\Ucc}{\mathcal{U}^{(2)}}
\newcommand{\Uc}{\mathcal{U}}
\newcommand{\Vc}{\mathcal{V}}

\newcommand{\sym}{\mathop{\mathrm{sym}}}
\newcommand{\skw}{\mathop{\mathrm{skw}}}
\newcommand{\veps}{\varepsilon}

\newcommand{\TT}{T^{(2)}}
\newcommand{\UU}{U^{(2)}}
\newcommand{\VV}{V^{(2)}}
\newcommand{\bb}{b^{(2)}}

\newcommand{\DD}{\mathcal{D}}

\newcommand{\DE}{\mathcal{D}}
\newcommand{\BAE}{\mathcal{A}}

\newcommand{\gpi}{\varPi_{\mathrm{grad}}}
\newcommand{\dpi}{\varPi_{\mathrm{div}}}

\newcommand{\Babuska}{Babu{\v{s}}ka}

\newcommand{\tr}{\mathrm{tr}}

\newcommand{\Hdiv}[1]{H(\mathrm{div},#1)}

\newcommand{\ip}[1]{\langle {#1} \rangle}
\newcommand{\jmpn}[1]{\big\| [#1] \big\|_{\partial\oh}}

\newcommand{\om}{\Omega}
\newcommand{\oh}{{\Omega_h}}
\newcommand{\EE}{E^{(2)}}
\newcommand{\ee}{e^{(2)}}
\newcommand{\xx}{x^{(2)}}
\newcommand{\yy}{y^{(2)}}

\newcommand{\sgn}{\hat{\sigma}_n}
\newcommand{\sgnh}{\hat{\sigma}_{n,h}}
\newcommand{\etah}{\hat{\eta}_{n,h}}

\newcommand{\dive}{\ensuremath{\mathop{\mathrm{div}}}}
\newcommand{\grad}{\ensuremath{\mathop{\mathrm{grad}}}}

\newcommand{\trace}{\ensuremath{\mathop{\mathrm{tr}}}}
\newcommand{\XXX}{\mathbb{X}}
\newcommand{\VVV}{\mathbb{V}}
\newcommand{\RRR}{\mathbb{R}}
\newcommand{\SSS}{\mathbb{S}}
\newcommand{\MMM}{\mathbb{M}}
\newcommand{\KKK}{\mathbb{K}}

\newcommand{\optnn}[1]{\left\|{#1}\right\|_{\mathrm{opt},\VV}}
\newcommand{\optnorm}[2]{\left\|{#1}\right\|_{\mathrm{opt},{#2}}}
\newcommand{\pt}{p_2}

\begin{document}

\title{A locking-free $hp$ DPG method for linear elasticity with symmetric stresses
}

\titlerunning{A DPG method for elasticity}        

\author{J.~Bramwell, L.~Demkowicz, J.~Gopalakrishnan, W.~Qiu$^{*}$
}


\institute{Jamie Bramwell \at
              The Institute for Computational Engineering and Sciences, \\
              the University of Texas at Austin, Austin, Texas 78712\\
              \email{jbramwell@ices.utexas.edu}           
           \and
           Leszek Demkowicz \at
              The Institute for Computational Engineering and Sciences, \\
              the University of Texas at Austin, Austin, Texas 78712\\
              \email{leszek@ices.utexas.edu}
           \and
           Jay Gopalakrishnan \at
              Department of Mathematics and Statistics, \\
              Portland State University, Portland OR 97207-0751\\
              \email{gjay@pdx.edu}   
           \and   
           *Corresponding author: Weifeng Qiu \at
              The Institute for Mathematics and its Applications, \\
              University of Minnesota, Minneapolis, Minnesota 55455\\
              \email{qiuxa001@ima.umn.edu}   
}

\date{Received: date / Accepted: date}

\maketitle

\begin{abstract}
  We present two new methods for linear elasticity that simultaneously
  yield stress and displacement approximations of optimal accuracy in
  both the mesh size $h$ and polynomial degree~$p$.  This is achieved
  within the recently developed discontinuous Petrov-Galerkin (DPG)
  framework. In this framework, both the stress and the displacement
  approximations are discontinuous across element interfaces. We study
  locking-free convergence properties and the interrelationships
  between the two DPG methods.
\end{abstract}

\section{Introduction}

In this paper we propose a new method for numerically solving the
system of equations describing linear elasticity.  The accurate
computation of stresses is of critical importance in many
applications. Yet, many traditional methods only yield approximations
to the displacement. This means that stress approximations must be
recovered afterward by numerical differentiation.  There are newer
methods, of the mixed and discontinuous Galerkin (DG) category, which
do give direct stress approximations. However, their stability
properties, as a function of both $h$ (the mesh size) and $p$ (the
polynomial degree of solution approximations) are presently
unknown. In this contribution, we bring to the table, a method of the
novel discontinuous Petrov-Galerkin (DPG) type, which exhibits
stability independent of $h$ and~$p$. The new method is the first
$hp$-optimal method for linear elasticity that can simultaneously
approximate the stress and the displacement. We are also able to show,
theoretically and practically, that the convergence of the discrete
solution does not deteriorate as the Poisson ratio approaches $0.5$,
i.e., the method does not lock.

To understand how the new DPG method sidesteps the traditional
difficulties, let us review the usual difficulties in designing
schemes that give direct stress approximations. Mixed
methods~\cite{ArnolAwanoWinth08,ArnolWinth02} based on the
Hellinger-Reissner variational principle face the problem of designing
an approximation space for stresses consisting of matrix functions
that are pointwise symmetric. While this is not difficult in itself,
when combined with two other additional requirements, difficulties
arise.  The first requirement is that forces on a mesh face shared by
two mesh elements must be in equilibrium.  The second requirement is
that the conforming stress space, together with a space for
displacement approximations, form a stable pair.  To put this in
mathematical terms, let $\MMM$ denote the space of real $N\times N$
matrices and let $\SSS$ denote its subspace of symmetric matrices.
The above mentioned stress properties imply that the exact stress
$\sigma$ on an elastic body occupying $\om \subseteq \RRR^N$ lies in
the space
\begin{equation}
  \label{eq:18}
\Hdiv{\om;\SSS} = \{ \sigma \in L^2(\om;\SSS): \dive \sigma \in
L^2(\om,\RRR^N)\}.  
\end{equation}
(Here, the set of functions from $\Omega$ into $\XXX$ whose components
are square integrable on $\om$ is denoted by $L^2(\Omega,\XXX)$, for
$\XXX=\SSS,\MMM,\RRR^N$ etc.) Mixed methods must use conforming finite
element subspaces of $\Hdiv{\om;\SSS}$.  Although such spaces are
known~\cite{ArnolAwanoWinth08}, they have too many unknowns (e.g.,
their lowest order space has 162 degrees of freedom on a single
tetrahedral element). Such rich spaces seem to be necessary to satisfy
both the first (conformity) and the second (discrete stability)
requirement. In contrast, our new DPG methods change the game by
separating the approximation and the stability issues.

The DPG method uses a {\em weaker} variational formulation for the
same problem. In this formulation, $\sigma$ is sought in the space
$L^2(\om, \MMM)$, in contrast to the space $\Hdiv{\om;\SSS}$ above.
Since $L^2(\om, \MMM)$, has no interelement continuity constraints, we
are able to design an approximating finite element subspace trivially.
Furthermore, due to the nonstandard stabilization mechanism of the DPG
scheme, the discrete stress space can be chosen to be a subspace of
$L^2(\om, \SSS)$, i.e., the method gives stresses that are exactly
(point-wise) symmetric. One can equally well choose stress
approximations in a subspace of $\Hdiv{\om;\SSS}$, disregarding
discrete stability considerations.  The stability of the DPG method is
inherited from the well-posedness of the new ultraweak formulation.
Of course, it is by no means trivial to prove this well-posedness (and
most of the analysis in this paper is devoted to it). It is provable
by adopting a Petrov-Galerkin framework where trial and test spaces
are different. For any given trial space, we can locally obtain a test
space that is guaranteed to yield stability.

Test spaces that guarantee stability can be obtained by following the
DPG methodology introduced
in~\cite{DemkoGopal:2010:DPG1,DemkoGopal:2010:DPG2}.  Our initial idea
was as follows: If one uses DG spaces, then given any test space norm,
one can {\em locally} construct test spaces that yield solutions that
are the best approximations in a dual norm on the trial space. This
dual norm was called the ``energy norm''. However, we
realized~\cite{DemkoGopalNiemi:2010:DPG3} that these energy norms are
often complicated to work with once we move beyond one-dimensional
problems. But we turned the tables
in~\cite{DemkoGopal:DPGanl,ZitelMugaDemko:2010:DPG4}, by showing that
given a desirable norm in which one wants the DPG solutions to
converge, there is a way to calculate the matching test space norm. We
refer to this norm as the ``optimal norm'' on the test space (see
\S~\ref{ssec:optimal}).  The catch is that the optimal norm is
nonlocal. Its use would make the computation of a basis for the test
space too expensive. Hence, we have been in pursuit of norm
equivalences.  If one uses, in place of the optimal norm, an
equivalent, but localizable test norm, then the DPG method, instead of
delivering the very best approximation, delivers a quasioptimal
approximation, i.e., the discretization error is bounded by a scalar
multiple of the best approximation error. This approach was applied to
a one-dimensional wave propagation problem
in~\cite{ZitelMugaDemko:2010:DPG4}. A number of further theoretical
tools were needed to develop an error analysis for multidimensional
problems. These, in the context of the simple Poisson equation, appear
in~\cite{DemkoGopal:DPGanl}. In this paper, we further generalize
these tools to the case of the elasticity problem and introduce new
tools to prove locking-free estimates.

Before we proceed to the details of this DPG method, let us mention
several alternative solutions to handle the difficulty of constructing
approximating subspaces of~\eqref{eq:18}. One approach is to relax the
symmetry constraint of stresses by using a Lagrange multiplier. This
means that $\sigma$ is sought in
\[
\Hdiv{\om;\MMM} = \{ \sigma \in L^2(\om;\MMM): \dive \sigma \in
L^2(\om,\RRR^N)\}  
\]
(cf.~\eqref{eq:18}), a space for which finite elements are easier to design. 
This avenue gave rise to mixed methods with weakly imposed symmetry~\cite{ArnolFalkWinth07,CockbGopalGuzma10,GopalGuzma10b,QiuDemko09,QiuDemko11,Stenb88}. 
Yet another avenue is to
keep the stress symmetry, but relax the
$H(\mathrm{div})$-conformity. This yielded non-conforming methods,
e.g.,~\cite{ArnoldWintherNC,GopalGuzma10a,ManHuShi09}.  However, none
of these methods have been proved to be $hp$-optimal. The closest
attempt to an $hp$-method is~\cite{QiuDemko11} which studies a
variable degree mixed method, but does not show how the error
estimates depend on~$p$.  In contrast, we will prove that the (two)
DPG methods we present in this paper are $hp$-optimal.  Furthermore,
since the DPG method can be reinterpreted as a least squares method in
a nonstandard inner product, it yields matrix systems that are
symmetric and positive definite (despite having both stress and
displacement as unknowns).

As mentioned above, there are two new DPG formulations in this
paper. The first is easy to derive and is a natural extension of our
work on the Poisson equation in~\cite{DemkoGopal:DPGanl}. The second
differs from the first due to the presence of a scaled Lagrange
multiplier. The multiplier serves to obtain the extra stability
required to prove the locking-free convergence estimates.

In the next section, we introduce the first DPG method. We also state
the main convergence result for the first method. The second method
and its convergence result is presented in
Section~\ref{sec:second}. Then, in Section~\ref{sec:relation}, we
study the relationship between these two methods, discovering when
they are equivalent. The proofs of the above mentioned two convergence
theorems appear in Section~\ref{sec:anl}. As corollaries to these
convergence theorems, we obtain $h$ and $p$ convergence rates in
Section~\ref{sec:rates}. In Appendix~\ref{sec:weakly-symm}, we present
a result on a mixed method that we crucially use in our proofs.

\section{The first DPG method}         \label{sec:first}

In this section, we present the derivation of the first of our two DPG
methods for linear elasticity. We also state a convergence theorem,
which will be proved in a later section.

Linear elasticity is described by two equations. The first is the
constitutive equation
\begin{subequations}
\label{eq:BVP}
  \begin{align}
    \label{eq:BVP1}
    A\sigma    & = \veps(u)\\
\intertext{and the second is the equilibrium equation}
    \label{eq:BVP2}
  \dive \sigma & = f. 
  \end{align}
  These equations are imposed on a domain $\om\subseteq \RRR^N$ and
  the space dimension $N$ equals $2$ or $3$.  We assume that $\om$ is
  a bounded open subset of $\RRR^N$ with connected
  Lipschitz boundary.  The stress $\sigma(x)$ is a function taking
  values in $\SSS$
  and its divergence ($\dive \sigma$) is taken row-wise.  The strain tensor is
  denoted by $ \veps(u)= (\grad u + (\grad u)')/2 = \sym\grad u$ where the prime
  ($'$) denotes matrix transpose, and $\sym M = (M+M')/2$.  The material properties are
  incorporated through the compliance tensor $A(x)$ in~\eqref{eq:BVP1}
  which at each $x\in \om$, is a fourth order tensor mapping $\SSS$
  into $\SSS$. 
  The vector function $u : \om \mapsto \RRR^N$ denotes
  the displacement field engendered by the body force $f:\om \mapsto
  \RRR^N$. We consider the simple boundary condition
\begin{equation}
  \label{eq:BVP3}
  u = 0  \qquad \text{ on } \partial\om
\end{equation}
\end{subequations}
which signifies that the elastic body is clamped on the boundary
$\partial\om$. (We will remark on extending the method to other boundary
condition later -- see Remark~\ref{rem:bc}.)

To motivate the derivation of the first scheme, we multiply the
equations of~\eqref{eq:BVP} by test functions $\tau:\om \mapsto
\mathbb{S}$ and $v:\om\mapsto \RRR^N$, supported on a domain $K$. We
temporarily assume $\tau$ and $v$ to be smooth so we can integrate by
parts to get
\begin{subequations}
  \label{eq:intg}
\begin{align}
  (A\sigma,\tau)_K
  +
  (u,\dive\tau)_K
    -\langle u,\tau\,n\rangle_{1/2,\partial K}
  & =  0,
  \\
  (\sigma,\nabla v)_K
  -\langle v,\sigma\,n \rangle_{1/2,\partial K}
  & = (f,v)_K.
\end{align}
\end{subequations}
Here $n$ denotes the outward unit normal on the boundary of the domain
under consideration, $(\cdot,\cdot)_K$ denotes the integral over $K$
of an appropriate inner product (Frobenius, dot product, or scalar
multiplication) of its arguments, and
$\langle\cdot,l\rangle_{1/2,\partial K}$ denotes the action of a
functional $l\in H^{-1/2}(\partial K)$. Here and throughout, we use
standard Sobolev spaces without explanation, e.g., $H^1(\Omega,\RRR^N)
= \{ v \in L^2(\Omega,\RRR^N): \grad v \in L^2(\Omega, \MMM)\}$.

Now, assume that the domain $\om$ where the boundary value
problem~\eqref{eq:BVP} is posed, admits a disjoint partitioning into
open ``elements'' $K$, i.e., $\cup \{\bar K: \; K \in \om_h \} = \bar
\om.$ We need not assume that elements are of any particular shape,
only that the mesh elements $K\in \oh$ have Lipschitz and {\em
  piecewise planar} boundaries, i.e., in two space dimensions, $K$ is
a Lipschitz polygon, and in three space dimensions, $K$ is a Lipschitz
polyhedron. We now sum up the equations of~\eqref{eq:intg} element by
element to obtain
\begin{subequations}
  \label{eq:2}
\begin{align}
  (A\sigma,\tau)_\oh
  +
  (u,\dive\tau)_{\Omega_{h}}
 -\langle u,\tau\,n\rangle_{\partial\Omega_{h}} 
  & =  0
  \\
  (\sigma,\nabla v)_{\Omega_{h}}
  -\langle v,\sigma\,n 
  \rangle_{\partial\Omega_{h}}
  & = (f,v)_\oh
  \\
  (\sigma,q)_{\Omega_{h}}
  & = 0. 
\end{align}
\end{subequations}
Here, we have additionally imposed the symmetry of the stress tensor
by the last equation (where $q$ is a skew-symmetric matrix valued test
function on $\Omega$) and used the following notations:
\begin{equation}
\label{inner_products}
(r,s)_{\Omega_{h}}=\sum_{K\in\Omega_{h}}(r,s)_{K},\quad 
\langle w,l\rangle_{\partial\oh}=\sum_{K\in\Omega_{h}}
\langle w,l\rangle_{1/2,\partial K}.
\end{equation}
The notation $\partial\oh$ is used for the collection $\{ \partial K: \; K \in
\oh\}$.  Note that it will be clear from the context if differential
operators are calculated element by element or globally, e.g., $\dive$
in~\eqref{eq:2} is calculated piecewise, while in~\eqref{eq:18} it is
the global.


The equations of~\eqref{eq:2} motivate the following rigorous functional
framework for an ultraweak variational formulation.  We let the traces
of $u$ and $\sigma$ in the terms $\langle
u,\tau\,n\rangle_{\partial\Omega_{h}}$ and $\langle v,\sigma\,n
\rangle_{\partial\Omega_{h}}$ to be new unknowns, which we call the
{\em numerical trace} and the {\em numerical flux}, resp. The
ultraweak variational formulation seeks
$(\sigma,u,\hat{u},\hat{\sigma}_{n})$ in the {\em trial space}
\begin{subequations}  \label{eq:DPG1}
\begin{align}
 \label{trial_space}
U & = L^{2}(\Omega;\mathbb{M})\times L^{2}(\Omega;\mathbb{V})
\times H_{0}^{1/2}(\partial\Omega_{h};\mathbb{V})\times H^{-1/2}(\partial\Omega_{h};\mathbb{V}),
\end{align}
satisfying 
 \begin{equation}
    \label{bilinear_dpg}
    b\big(\,(\sigma,u,\hat{u},\hat{\sigma}_{n}),\; 
    (\tau,v,q)\,\big)
    =l(\tau,v,q)\qquad\quad \forall\, (\tau,v,q)\in V, 
  \end{equation}
  where the {\em test space} $V$ is defined by 
\begin{align}
\label{test_space}
V & = H(\text{div},\Omega_{h};\mathbb{S})\times H^{1}(\Omega_{h};\mathbb{V})\times 
L^{2}(\Omega_{h};\mathbb{K}),
\end{align}
where $\KKK$ denotes the subspace of $\MMM$ consisting of all
skew-symmetric matrices, and the bilinear form $b(\cdot,\cdot)$ and
the linear form $l(\cdot)$ are motivated by~\eqref{eq:2}. Namely,
 \begin{align}
    \label{binear_form_concrete}
    b((\sigma,u,\hat{u},\hat{\sigma}_{n}),(\tau,v,q))
    & = 
    (A\sigma,\tau)_{\Omega_{h}}+(u,\dive\tau)_{\Omega_{h}}
    -\langle\hat{u},\tau\,n\rangle_{\partial\Omega_{h}} 
    \\\nonumber
    & \qquad +(\sigma,\nabla v)_{\Omega_{h}}-\langle v,\hat{\sigma}_{n}\rangle_{\partial\Omega_{h}}
    +(\sigma,q)_{\Omega_{h}}, \\\nonumber
    l(\tau,v,q) & = (f,v).
  \end{align}
\end{subequations}
In the notations of~\eqref{trial_space} and~\eqref{test_space}, 
 $\VVV$ denotes the vector space $\RRR^N$ and 
\begin{align*}
H_{0}^{1/2}(\partial\Omega_{h};\mathbb{V}) & =\lbrace\eta :\exists w\in H_{0}^{1}(\Omega;\mathbb{V})\text{ with }
\eta|_{\partial K}=w|_{\partial K},\forall K\in\Omega_{h}\rbrace,\\
H^{-1/2}(\partial\Omega_{h};\mathbb{V}) & =\lbrace\eta\in\smash[b]{\mathop{\Pi}_{K\in\oh}}H^{-1/2}(\partial K;\mathbb{V}):\exists
q\in H(\text{div},\Omega;\mathbb{M})\text{ with }
&
\\
&\qquad\qquad\qquad\qquad\eta|_{\partial K}=q\,n|_{\partial K},\forall K\in\Omega_{h}\rbrace,\\
H(\text{div},\Omega_{h};\mathbb{S}) & =\lbrace\tau:\tau|_{K}\in H(\text{div},K;\mathbb{S}),\forall K\in\Omega_{h}\rbrace,\\
H^{1}(\Omega_{h};\mathbb{V}) & =\lbrace v:v|_{K}\in H^{1}(K;\mathbb{V}),\forall K\in\Omega_{h}\rbrace.
\end{align*}
The norms on $H_{0}^{1/2}(\partial\Omega_{h};\mathbb{V})$ and $H^{-1/2}(\partial\Omega_{h};\mathbb{V})$ 
are defined by 
\begin{align}
\label{trace_norm}
\Vert\hat{u}\Vert_{H_{0}^{1/2}(\partial\Omega_{h})} = \inf\lbrace\Vert w\Vert_{H^{1}(\Omega)}:\forall
w\in H_{0}^{1}(\Omega;\mathbb{V}) \text{ with }\hat{u}-w|_{\partial K}=0 \rbrace,\\
\label{flux_norm}
\Vert\hat{\sigma}_{n}\Vert_{H^{-1/2}(\partial\Omega_{h})} = \inf\lbrace\Vert q\Vert_{H(\text{div},\Omega)}:
\forall q\in H(\text{div},\Omega;\mathbb{M})\text{ with }\hat{\sigma}_{n}-q\,n|_{\partial K}=0 \rbrace.
\end{align}
The  trial and test norms are defined by
\begin{subequations}
\begin{align}
\label{trial_norm}
\Vert(\sigma,u,\hat{u},\hat{\sigma}_{n})\Vert_{U}^{2} & =\Vert\sigma\Vert_{\Omega}^{2}+\Vert u\Vert_{\Omega}^{2}
+\Vert\hat{u}\Vert_{H_{0}^{1/2}(\partial\Omega_{h})}^{2}
+\Vert\hat{\sigma}_{n}\Vert_{H^{-1/2}(\partial\Omega_{h})}^{2},\\
\label{test_norm}
\Vert(\tau,v,q)\Vert_{V}^{2} & =\Vert\tau\Vert_{H(\text{div},\Omega_{h})}^{2}+
\Vert v\Vert_{H^{1}(\Omega_{h})}^{2}+\Vert q\Vert_{\Omega}^{2}.
\end{align}
\end{subequations}
Here the norms on $H^{1}(\Omega_{h};\mathbb{V})$ and
$H(\text{div},\Omega_{h};\mathbb{S})$ are defined by
\begin{align}
\label{broken_norm1}
\Vert v\Vert_{H^{1}(\Omega_{h})}^{2}=(v,v)_{\Omega_{h}}+(\grad v,\grad v)_{\Omega_{h}},\\
\label{broken_norm2}
\Vert \tau\Vert_{H(\text{div},\Omega_{h})}^{2}=(\tau,\tau)_{\Omega_{h}}+(\dive\tau,\dive\tau)_{\Omega_{h}}.
\end{align}
This completes the description of our new infinite dimensional DPG
variational formulation.

\begin{remark} \label{rem:ext1} 
  Like in many other numerical formulations, in the DPG
  formulation~\eqref{eq:DPG1}, we need to the extend the domain of the
  compliance tensor $A(x)$ from $\SSS$ to $\MMM$. There are many ways
  to perform this extension.  To choose one, 
  decompose $\MMM$
  orthogonally (in the Frobenius inner product) into $\KKK$ and
  $\SSS$.  A standard way to extend $A(x)$ from $\SSS$ to $\MMM$ is to
  define $A(x) \kappa = \kappa$ for all $\kappa$ in $\KKK$. Then,
  whenever the original $A(x)$ is self-adjoint and positive definite
  on $\SSS$, the extended $A(x)$ is also self-adjoint and positive
  definite on $\MMM$.  
\end{remark}

We assume throughout that $A(x)$ (i.e., its above mentioned extension) is
self-adjoint and positive definite uniformly on $\om$. We
also assume that the components of (the extended) $A$ are in
$L^\infty(\om)$.

Next, we describe the discrete DPG scheme. This is done following
verbatim the abstract setup in~\cite[Section~2]{DemkoGopal:DPGanl}
(see also~\cite{DemkoGopal:2010:DPG2}). Accordingly, we define the
{\em trial-to-test} operator
$T: U \mapsto V$ by
\begin{equation}
  \label{eq:T}
  (T \Uc, \Vc)_V = b(\Uc,\Vc),\qquad \forall  \,\Vc \in V
  \text{ and } \forall\,\Uc \in U.
\end{equation}
We select {\em any} finite dimensional subspace $U_h \subseteq U$ and
set the corresponding finite  dimensional test space by 
\[
V_h = T (U_h).
\]
Then the DPG approximation $(\sigma_h,u_h,\hat u_h,\sgnh) \in U_h$ satisfies 
\begin{equation}
  \label{eq:dpgapprox}
  b\big(\,(\sigma_h,u_h,\hat u_h,\sgnh) ,\; 
  (\tau,v,q)\,\big)
    =l(\tau,v,q)\qquad\quad \forall\, (\tau,v,q)\in V_h.
\end{equation}
The distance between this approximation and the exact solution can be
bounded as stated in the next theorem.

\begin{theorem}[Quasioptimality]
  \label{thm:dpg1}
  Let $U_h \subseteq U$. Then,~\eqref{bilinear_dpg} has a unique
  solution $(\sigma,u,\hat u,\sgn) \in U$ and~\eqref{eq:dpgapprox} has
  a unique solution $(\sigma_h,u_h,\hat u_h,\sgnh) \in U_h$.
  Moreover, there is a $C^{(1)}>0$ independent of the subspace $U_h$
  and the partition $\oh$ such that
  \[
  \DE \;\le\; C^{(1)}\; \BAE,
  \]
  where $\DE$ is the discretization error and $\BAE$ is the error in
  best approximation by $U_h$, defined by
  \begin{align*}
    \DE
    & =
    \| \sigma - \sigma_h \|_{L^2(\om)} 
    +
    \| u - u_h \|_{L^2(\om)}
    +
    \| \hat u - \hat u_h \|_{H_0^{1/2}(\partial\oh)}
    +
    \| \sgn - \sgnh \|_{H^{-1/2}(\partial\oh)}.
    \\
    \BAE
    & = \!\!\! \displaystyle{\inf_{ (\rho_h, w_h, \hat z_h, \etah) \in U_h }}
    \\
    & \bigg(    
    \| \sigma - \rho_h \|_{L^2(\om)} 
     + 
    \| u - w_h \|_{L^2(\om)}  
     + 
     \| \hat u - \hat z_h \|_{H_0^{1/2}(\partial\oh)}
     +  
    \| \sgn - \etah \|_{H^{-1/2}(\partial\oh)} 
    \bigg).
\end{align*}
\end{theorem}

This result is comparable to Cea's lemma in traditional finite element
theory. Of importance is the independence of $C^{(1)}$ with respect to
$U_h$. Specifically, we are interested in setting $U_h$ to $hp$-finite
element subspaces with extreme variations in $h$ and $p$ to capture
singularities or thin layers in solutions. In this case, the constant
$C^{(1)}$, being independent of $U_h$, is independent of {\em both}
the mesh size $h$ and the polynomial degree~$p$. As such, this forms
the first method for linear elasticity with provably $hp$-optimal
convergence rates of the same order for $\sigma$ and $u$.  Although
several mixed methods yielding good approximations to $\sigma$ are
known, proving their $hp$-optimality requires proving an inf-sup
condition carefully tracking the dependence of constants on~$p$, a
feat yet to be achieved. For a proof of Theorem~\ref{thm:dpg1}, see
Section~\ref{sec:anl}.

\section{The second DPG method}       \label{sec:second}

The robustness of numerical methods with respect to the Poisson ratio
is an important consideration in computational mechanics. Methods that
are not robust exhibit {\em locking}.  Note that we did not assert in
Theorem~\ref{thm:dpg1} that the constant $C^{(1)}$ is independent of
the Poisson ratio. However, in all our numerical experiments (see
Section~\ref{sec:numerical}), the method showed locking-free
convergence. This section serves as a first step towards explaining
this locking-free behavior theoretically.

The second DPG method given below is designed so that we can establish
locking-free convergence with respect to the Poisson ratio.  It has
one more trial variable. In this section, we will assert its
locking-free convergence properties, restricting ourselves to
isotropic materials. In the next section, we will provide a sufficient
condition under which the first and the second DPG methods are
equivalent. This gives theoretical insight into the locking-free
behavior of {\em both} the first and the second methods.

Let us begin by defining the essential infimum
\begin{equation}
\label{Q0_def}
Q_{0} = \essinf_{x\in\Omega} \big( \trace (A(x)I)\, \big).
\end{equation}
Obviously, $Q_{0}>0$ for an isotropic material with Poisson ratio
$\nu<0.5$.  The second method is motivated by the same integration by
parts as in~\eqref{eq:2}, but with the following additional
observation in mind: If we set $\tau=I$ in (\ref{eq:2}) and recall
that $u=0|_{\partial\Omega}$, then we have that
\begin{equation}
\label{zero_trace_eq}
\int_{\Omega}\trace A\sigma=0.
\end{equation}
Imposing this condition via a Lagrange multiplier, we obtain another
ultraweak formulation with the following bilinear and linear forms:
\begin{subequations}
  \label{eq:dpg2}
  \begin{align}
    \nonumber
  \bb((\sigma,u,\hat{u},\hat{\sigma}_{n},\alpha),(\tau,v,q,\beta))
  & = 
  (A\sigma,\tau)_{\Omega_{h}}+(u,\dive\tau)_{\Omega_{h}}
  -\langle\hat{u},\tau\,n\rangle_{\partial\Omega_{h}}+Q_{0}^{-1}(\alpha I,A\tau)_{\Omega_{h}} \\\nonumber
  & \qquad +(\sigma,\nabla v)_{\Omega_{h}}+(\sigma,q)_{\Omega_{h}}
  -\langle v,\hat{\sigma}_{n}\rangle_{\partial\Omega_{h}}\\
  \label{eq:6}
  & \qquad +Q_{0}^{-1}(A\sigma,\beta I)_{\Omega_{h}}
  \\ \nonumber
  l^{(2)}(\tau,v,q,\beta) & = (f,v).
\end{align}
Here  $I$ is the identity matrix.
The trial space is now set to 
\begin{equation}
  \label{eq:7}
  \UU = L^{2}(\Omega;\mathbb{M})\times
  L^{2}(\Omega;\mathbb{V}) \times
  H_{0}^{1/2}(\partial\Omega_{h};\mathbb{V}) \times
  H^{-1/2}(\partial\Omega_{h};\mathbb{V})\times \mathbb{R}
\end{equation}
and the test space is set to
\begin{equation}
  \label{eq:9}
  \VV  = H(\text{div},\Omega_{h};\mathbb{S})\times
  H^{1}(\Omega_{h};\mathbb{V})\times L^{2}(\Omega_{h};\mathbb{K})\times \mathbb{R}.
\end{equation}
\end{subequations}
By (\ref{zero_trace_eq}), the solution $(\sigma,u)$ of (\ref{eq:BVP}) 
will be the solution of (\ref{eq:dpg2}) with 
\begin{equation*}
\hat{u}=u|_{\partial\Omega_{h}},\quad \hat{\sigma}_{n}=\sigma n |_{\partial\Omega_{h}},
\quad\text{and}\quad \alpha = 0.
\end{equation*} 
However, more work is needed to conclude similar statements
at the discrete level (see the next section).

\begin{remark} \label{rem:bc}
  In the case of mixed boundary conditions we must add the term
  \begin{equation}
    \label{eq:bcterm}
      -Q_{0}^{-1}\langle \hat{u},(\beta I)n\rangle_{\partial\Omega_{h}}
  \end{equation}
  to the the expression~\eqref{eq:6}. Note that this term vanishes in
  the case of kinematic boundary conditions analyzed in this
  paper. However, for more general boundary conditions, $\hat u$ can
  be be nonzero on the parts of the boundary where traction conditions
  are imposed. Hence~\eqref{eq:bcterm} simplifies to a boundary
  integral that is nonzero in general.
\end{remark}

The second DPG method is obtained by constructing a discrete scheme as
before from the ultraweak formulation
(following~\cite[Section~2]{DemkoGopal:DPGanl}).  The trial-to-test
operator in this case (cf.~\eqref{eq:T}) is $\TT: \UU \mapsto \VV$ by
\begin{equation}
  \label{eq:T}
  (\TT \Uc, \Vc)_{\VV} = \bb(\Uc,\Vc),\qquad \forall  \,\Vc \in \VV.
\end{equation}
Let $U_h^{(2)} \subseteq U^{(2)}$ be any finite dimensional
subspace. We set $\VV_h = \TT (\UU)$. The second DPG approximation
$(\sigma_h,u_h,\hat u_h,\sgnh,\alpha_{h}) \in U_h^{(2)}$ satisfies 
\begin{equation}
  \label{eq:dpgapprx2}
    \bb\big( \, (\sigma_h,u_h,\hat u_h,\sgnh,\alpha_{h}),
    \;(\tau,v,q,\beta)\,\big)
    = l((\tau,v,q,\beta))
    \qquad\forall\, (\tau,v,q,\beta)\in \VV_h.
\end{equation}
As in the case of the first DPG method, we are able to prove a
quasioptimality result (see the next theorem) bounding the discretization error
\begin{align*}
    \DE^{(2)} 
    & =
    \| \sigma - \sigma_h \|_{L^2(\om)} 
    +
    \| u - u_h \|_{L^2(\om)}+\vert\alpha - \alpha_{h}\vert\\
    & \qquad +
    \| \hat u - \hat u_h \|_{H_0^{1/2}(\partial\oh)}
    +
    \| \sgn - \sgnh \|_{H^{-1/2}(\partial\oh)}.
\end{align*}
by the error in best approximation
\begin{equation*}
  \begin{gathered}
    \BAE^{(2)}
     = \!\!\! \displaystyle{\inf_{ (\rho_h, w_h, \hat z_h, \etah,\gamma_{h}) \in U_h }}
    \bigg(    
    \| \sigma - \rho_h \|_{L^2(\om)} 
     + 
    \| u - w_h \|_{L^2(\om)}+\vert \alpha-\gamma_{h}\vert  
    \\
    \qquad\qquad\qquad\qquad\qquad\qquad +\;  
    \| \hat u - \hat z_h \|_{H_0^{1/2}(\partial\oh)}
     +  
    \| \sgn - \etah \|_{H^{-1/2}(\partial\oh)} 
    \bigg).
  \end{gathered}
\end{equation*}
However, we are also able to prove a stronger result under the
following assumption.

\begin{assumption}[Isotropic material]
  \label{asm:iso}
  We assume that
\begin{equation}
\label{PQ_decomp}
A\tau = P\tau_{D}+Q\dfrac{\trace (\tau)}{N}I
\end{equation}
where
\[
\tau_{D}=\tau-\dfrac{\trace (\tau)}{N}I,
\]
for any $\tau$ in $\MMM$, and $P$ and $Q$ are positive scalar
functions on $\Omega$. (Then, obviously $Q\geq Q_{0}$ for the $Q_0$
defined in~\eqref{Q0_def}.) When~\eqref{PQ_decomp} holds, we also
define
\begin{align} 
\label{eq:B}
B &= Q_{0}^{-1}\Vert Q\Vert_{L^{\infty}(\Omega)},
\\
\label{eq:P0}
P_{0}
& = \essinf_{x\in \Omega} P(x).
\end{align}
(Note that~\eqref{PQ_decomp} is assumed to hold for all $\tau\in \MMM$.)
\end{assumption}

\begin{remark}
  \label{rem:extiso}
  When we consider isotropic materials, we do not extend $A$ from
  $\SSS$ to $\MMM$ in the way suggested in
  Remark~\ref{rem:ext1}. Instead, we assume that it is extended from
  $\SSS$ to $\MMM$ by $A \kappa = P \kappa$ for all $\kappa \in
  \KKK$. This ensures that~\eqref{PQ_decomp} holds for all $\tau\in
  \MMM$, not just for all $\tau \in \SSS$.
\end{remark}

\begin{theorem}[Quasioptimality of the second DPG method]
  \label{thm:dpg2}
  Let $U_h^{(2)} \subseteq \UU$. Then,~\eqref{eq:dpg2} has a unique
  solution $(\sigma,u,\hat u,\sgn,\alpha) \in U^{(2)}$
  and~\eqref{eq:dpgapprx2} has a unique solution $(\sigma_h,u_h,\hat
  u_h,\sgnh,\alpha_{h}) \in U_h^{(2)}$. Moreover, there is a
  $C^{(2)}>0$ independent of the subspace $U_h^{(2)}$ and the
  partition $\oh$ such that
  \[
  \DE^{(2)} \;\le\; C^{(2)}\; \BAE^{(2)}.
  \]
  If in addition, Assumption~\ref{asm:iso} holds, then $C^{(2)}$ can be
  chosen to be
  \begin{equation}
  \label{uniform_stability_constant}
  C^{(2)} = \bar{c}P_{0}^{-1}(\Vert A\Vert+B)^{3}B^{4}(\Vert A\Vert+P_{0}+1)^{2},
  \end{equation}
  a constant independent of $Q_0$, and consequently the method does
  not lock.  (Here, the positive constant $\bar{c}$ is independent of
  $A$.)
\end{theorem}

The proof of this theorem appears in Section~\ref{sec:anl}.

\section{The relationship between the two DPG methods}   \label{sec:relation}

In this section we will establish that for homogeneous isotropic
materials the two DPG methods are equivalent. We will also show that
despite the additional trial variable, the second DPG method can be
solved in essentially the same cost as  the first.

\subsection{The equivalence}

Recall that the first DPG variational formulation
uses the trial space 
\[U = L^{2}(\Omega;\mathbb{M})\times L^{2}(\Omega;\mathbb{V}) \times
H_{0}^{1/2}(\partial\Omega_{h};\mathbb{V})\times
H^{-1/2}(\partial\Omega_{h};\mathbb{V}),
\] 
while the second uses $U \times \RRR$. We begin with the following simple lemma.

\begin{lemma} \label{lem:relation_lemma1} 
  For any $\Uc \equiv (\sigma,u,\hat u, \sgn)$ in $U$, let 
  \[
  (\tau,v,q) = T\Uc.
  \]
  Then, with $\Ucc \equiv (\sigma,u,\hat u, \sgn,0)$, 
  \[
  \TT \Ucc =  (\tau, v, q, \beta), 
  \]
  with $\beta =Q_{0}^{-1}(A\sigma,I)_{\Omega}$.
\end{lemma}
\begin{proof}
  By the definition of $T$, we have, for any $(\delta_\tau,\delta_
  v,\delta_q)\in V=H(\text{div},\Omega_{h};\mathbb{S})\times
  H^{1}(\Omega_{h};\mathbb{V})\times L^{2}(\Omega_{h};\mathbb{K})$, 
\begin{align*}
\big(\,(\tau,v,q),\;  (\delta_\tau,\delta_v,\delta_q) \,\big)_{V}
& = 
(\tau,\delta_\tau)_{\Omega}
 +(\dive\tau,\delta_\tau)_{\Omega_{h}}+(v,\delta_v)_{\Omega}+(\nabla v,\nabla\delta_v)_{\Omega_{h}}+(q,\delta_q)_{\Omega}\\
=(A\sigma,\delta_\tau)_{\Omega} 
& +(u,\dive\delta_\tau)_{\Omega_{h}} 
-\langle \hat{u},\delta_\tau n\rangle_{\partial\Omega_{h}}
+(\sigma,\nabla \delta_v)_{\Omega_{h}}
%
-\langle \delta_v,\hat{\sigma}_{n}\rangle_{\partial\Omega_{h}}
+(\sigma,\delta_q)_{\Omega}.
\end{align*}
Therefore choosing
\begin{equation*}
\beta=Q_{0}^{-1}(A\sigma,I)_{\Omega},
\end{equation*}
we obviously  obtain
\begin{align*}
\big(\,(\tau,v,q,\beta), &\;  
 (\delta_\tau,\delta_v,\delta_q,\delta_\beta) \,\big)_{\VV}
\\
&= (\tau,\delta_\tau)_{\Omega}
+(\dive\tau,\delta_\tau)_{\Omega_{h}}+(v,\delta_v)_{\Omega}
+(\nabla v,\nabla\delta_v)_{\Omega_{h}}
+ (q,\delta_q)_{\Omega}
+\beta\delta_\beta
\\
& =(A\sigma,\delta_\tau)_{\Omega}
+(u,\dive\delta_\tau)_{\Omega_{h}}-\langle \hat{u},\delta_\tau n\rangle_{\partial\Omega_{h}}
+(\sigma,\nabla\delta_v)_{\Omega_{h}}
%
-\langle \delta_v,\hat{\sigma}_{n}\rangle_{\partial\Omega_{h}}
\\
& 
%
+(\sigma,\delta_q)_{\Omega}
+Q_{0}^{-1}(A\sigma,\delta_\beta I)_{\Omega}
\\
& =   
\bb\big(\,(\sigma,u,\hat{u},\hat{\sigma}_{n},0),\;(\tau,v,q,\beta)\big)
\end{align*}
for any $(\delta_\tau,\delta_v,\delta_q,\delta_\beta)\in V\times\mathbb{R}=V^{(2)}$.
This finishes the proof.
\end{proof}

We use the above result together with the following assumption to
prove the equivalence.  The assumption essentially states that the
material is homogeneous and isotropic and that the discrete trial
space at least contains two specific functions related to the identity
matrix.

\begin{assumption} \label{ass_two_methods} 
  Suppose Assumption~\ref{asm:iso} (on isotropy) holds.  Let the discrete trial
  subspace $U_h \subseteq U$ of the first DPG method be
\begin{equation}
  \label{eq:Uh}
  U_h = \Sigma_h \times W_h \times M_h \times F_h,
\end{equation}
and let the second DPG method use the trial space $U_h \times
\RRR$. We assume that
\begin{enumerate}
\item $Q(x)=Q_{0}$ for all $x\in \Omega$,
\item $I \in \Sigma_h$, 
\item $In|_{\partial\Omega_h} \in F_h$.
\end{enumerate}
\end{assumption}

\begin{lemma}
\label{lem:relation_lemma2}
If Assumption~\ref{ass_two_methods} holds, then $(I,0,0)\in V_{h}
\equiv T(U_h)$.
\end{lemma}

\begin{proof}
  By virtue of the assumption, the trial function $\Uc = (\sigma,
  u,\hat u,\sgn)$ with $\sigma = I, u= 0, \hat{\sigma}_{n} =
  In|_{\partial\Omega_{h}},$ and $\hat{u}=0$, is in $U_h$. Hence,
  $({\tau},{v},{q}) \equiv T\Uc$ is in $V_h$ and satisfies
  \begin{equation}
    \label{eq:8}
    \begin{aligned}
      ({\tau},\delta_\tau)_{\Omega}
      & +(\dive{\tau},\dive\delta_\tau)_{\Omega_{h}}+({v},\delta_v)_{\Omega}
      +(\nabla{v},\nabla\delta_v)_{\Omega_{h}}+({q},\delta_q)_{\Omega}\\
      &= (AI,\delta_\tau)_{\Omega}+(I,\nabla\delta_v)_{\Omega_{h}}-\langle\delta_v, In\rangle_{\partial\Omega_{h}}
      +(I,\delta_q)_{\Omega}
    \end{aligned}
  \end{equation}
  for all $(\delta_\tau,\delta_v,\delta_q) \in V =
  H(\text{div},\Omega_{h};\mathbb{S})\times
  H^{1}(\Omega_{h};\mathbb{V}) \times L^{2}(\Omega_{h};\mathbb{K})$.
  The last term in~\eqref{eq:8} vanishes due to the skew symmetry of
  $\delta_q$. Integration by parts shows that
  $(I,\nabla\delta_v)_{\Omega_{h}}-\langle\delta_v,
  In\rangle_{\partial\Omega_{h}}=0$ also. Hence, we conclude that
  $({\tau},{v},{q})=(Q_{0}I,0,0)$ is the unique solution
  of~\eqref{eq:8}. Since $({\tau},{v},{q})$ is in $V_h$, we have
  proved the lemma.
\end{proof}

\begin{lemma} \label{lem:relation_lemma3} 
  If Assumption~\ref{ass_two_methods} holds, then
  \[
  (I,0,0,0) = \TT(0,0,0,0,1).
  \]
\end{lemma}
\begin{proof}
  Let  $({\tau},{v},{q},{\beta}) = \TT(0,0,0,0,\alpha)$.
  Then by the definition of $\TT$, we have
\begin{align*}
({\tau},\delta_\tau)_{\Omega}
& 
+(\dive{\tau},\dive\delta_\tau)_{\Omega_{h}}+({v},\delta_v)_{\Omega}
+(\nabla{v},\nabla\delta_v)_{\Omega_{h}}+({q},\delta_q)_{\Omega}
+{\beta}\delta_\beta
\\
&
= Q_{0}^{-1}(\alpha I, A \delta_\tau)_{\Omega}
\end{align*}
for all $(\delta_\tau,\delta_v,\delta_q,\delta_\beta)\in V^{(2)}=H(\text{div},\Omega_{h};\mathbb{S})\times 
H^{1}(\Omega_{h};\mathbb{V})\times L^{2}(\Omega_{h};\mathbb{K})\times\mathbb{R}$. 
Putting $\alpha=1$, 
and using
the symmetry of $A$, the right hand side 
$Q_{0}^{-1}(\alpha I, A \delta_\tau)_{\Omega}
= Q_{0}^{-1}( AI, \delta_\tau)_{\Omega}$, which due to
 the assumption on $A$ equals
$ (I,\delta_\tau)_{\Omega}$, i.e., 
\begin{align*}
({\tau},\delta_\tau)_{\Omega}
& 
+(\dive{\tau},\dive\delta_\tau)_{\Omega_{h}}
= (I,\delta_\tau)_{\Omega}.
\end{align*}
It is now obvious that $({\tau},{v},{q},{\beta})=(I,0,0,0)$. 
\end{proof}

\begin{theorem}  \label{thm:relation_thm1}
  Suppose Assumption~\ref{ass_two_methods} holds.  Then
  $(\sigma_{h},u_{h},\hat{u}_{h},\hat{\sigma}_{n,h})$ solves the 
  first DPG method~\eqref{eq:dpgapprox} if and only if
  $(\sigma_{h},u_{h},\hat{u}_{h},\hat{\sigma}_{n,h},0)$ is the
  discrete solution of the second DPG method~\eqref{eq:dpgapprx2}.
\end{theorem}
\begin{proof}
  Suppose $(\sigma_{h},u_{h},\hat{u}_{h},\hat{\sigma}_{n,h})$ solves
  the first DPG method~\eqref{eq:DPG1}. To show that
  the function $(\sigma_{h},u_{h},\hat{u}_{h},\hat{\sigma}_{n,h},0)$ satisfies the
  second DPG method~(\ref{eq:dpg2}), we have to show that the equations
  \begin{subequations}
    \label{eq:10}
  \begin{align}
    \label{eq:10-a}
    (A\sigma_h,\tau)_{\Omega_{h}}+(u_h,\dive\tau)_{\Omega_{h}}
    -\langle\hat{u},\tau\,n\rangle_{\partial\Omega_{h}}
    +Q_{0}^{-1}(\alpha I,A\tau)_{\Omega_{h}}
      & = 0 
      \\
      \label{eq:10-b}
      (\sigma_h,\nabla v)_{\Omega_{h}}
      -\langle v,\sgnh \rangle_{\partial\Omega_{h}}
      & = (f,v)
      \\
      \label{eq:10-d}
      (\sigma_h,q)_{\Omega_{h}} & =0
      \\
      \label{eq:10-c}
      Q_{0}^{-1}(A\sigma,\beta I)_{\Omega_{h}} 
      & = 0,
  \end{align}
  \end{subequations}
  hold, with $\alpha=0$, for all $(\tau,v,q,\beta) \in \TT( U_h \times
  \RRR)$. To this end, we proceed in two steps.

  First, we consider test functions of the type $\TT( U_{h}\times
  \{0\})$. By virtue of Lemma~\ref{lem:relation_lemma1}, these test
  functions are in $T(U_h) \times \RRR$.  Hence by the fact that
  $(\sigma_{h},u_{h},\hat{u}_{h},\hat{\sigma}_{n,h})$ solves the
  equation of the first DPG method~\eqref{eq:dpgapprox} for all
  $(\tau,v,q) \in T(U_h)$, we observe that~\eqref{eq:10-a},
  \eqref{eq:10-b} and~\eqref{eq:10-d} hold. To show that~\eqref{eq:10-c} also holds,
  we observe that because of Lemma~\ref{lem:relation_lemma2}, we may
  put $(\tau,v,q)=(\beta I,0,0)$, for any $\beta \in \RRR$,
  in~\eqref{eq:dpgapprox} to get
  \begin{align*}
    (A\sigma_{h},\beta I)_{\Omega_{h}}
    +(u_{h},\dive \beta I)_{\Omega_{h}}
    -\langle\hat{u}_{h},\beta I n\rangle_{\partial\Omega_{h}}
   = 
   0.
  \end{align*}
  Multiplying by $Q_0^{-1}$ and simplifying, we obtain~\eqref{eq:10-c}
  for any $\beta$.

  Second, consider test functions of the type $\TT( \{(0,0,0,0)\}
  \times \RRR)$. But by Lemma~\ref{lem:relation_lemma3}, we know that 
  $\TT(0,0,0,0,1) =
  (I,0,0,0)$.  So to show that~\eqref{eq:10} holds for this test
  function, it suffices to prove that~\eqref{eq:10-a} holds with $\tau
  = I$. But this follows from Lemma~\ref{lem:relation_lemma2}, which
  shows that $(I,0,0)$ is in $V_h$.

  Conversely, suppose
  $(\sigma_{h},u_{h},\hat{u}_{h},\hat{\sigma}_{n,h},0)$ satisfies the
  second DPG method~\eqref{eq:dpgapprx2}. Then,
  $(\sigma_{h},u_{h},\hat{u}_{h},\hat{\sigma}_{n,h})$ must satisfy the
  first DPG method, for if not, there must be another function
  $(\sigma_{h}',u_{h}',\hat{u}_{h}',\hat{\sigma}_{n,h}')$ solving the
  first DPG method. But then, the already proved implication shows that
  $(\sigma_{h}',u_{h}',\hat{u}_{h}',\hat{\sigma}_{n,h}',0)$ must solve
  the second DPG method. This contradicts the unique solvability of
  the second DPG method asserted in Theorem~\ref{thm:dpg2}.
\end{proof}

\subsection{Solving the second DPG system}

The relationship between the two methods revealed above, can be
utilized to obtain a useful strategy for solving the second DPG
system.

To understand the linear systems that result from both methods, 
we assume that a basis for $U_h$ is made of local (standard finite
element) functions
\[
e_i = (\sigma_{i},u_{i},\hat{u}_{i},\hat{\sigma}_{n,i})
\]
for $i=1,2,\ldots, m$.  Then the corresponding basis for $V_h$ is
furnished by $t_j = Te_i$. Hence the $m\times m$ stiffness matrix of
the first DPG method $E$ has entries
\[
E_{ij} = b( e_j, t_i).
\]
It is easily seen that this matrix is symmetric and positive
definite. (This is a general property of DPG stiffness matrices --
see~\cite{DemkoGopal:DPGanl} or~\cite{DemkoGopal:2010:DPG2}.) In
addition, $E$ is sparse due to the locality of the basis functions.

Now, consider the stiffness matrix of the second DPG method. Here, we
need a basis for $U_h \times \RRR$. It is natural to take as basis for
$U_h \times \RRR$, the $m+1$ functions defined by
\[
\ee_i = \left\{
  \begin{aligned}
     &(e_i,0), && \text{ for all } i=1,2,\ldots, m,
     \\
     & (0,0,0,0,1), && \text{ for } i=m+1.
  \end{aligned}
\right.
\]
But we should then note that the stiffness matrix 
\[
\EE_{ij} = \bb( \ee_j, \TT \ee_i )
\]
is no longer sparse. This is because $\TT \ee_i$ is not locally
supported. Indeed, if we write $\TT \ee_i$ as $(\tau_i,
v_i,q_i,\beta_i)$, then $\beta_i$ can be globally supported even if
$\ee_i$ is local. Thus, although $\EE$ is symmetric and positive
definite, one may be led into concluding that the second DPG method is
too expensive due to the non-sparsity.

However, this is not the case. Below, we will show how to solve a
system $\EE \xx = \yy$ by solving a system $E x = y$ and performing a
few additional inexpensive steps. A key observation is that $\tilde E$
is a rank-one perturbation of $E$.

\begin{proposition}
  \label{prop:dpg2}
  Decompose the matrix $\EE$ into 
  \[
  \EE = 
  \begin{bmatrix}
    \tilde E & c \\
    c' & d 
  \end{bmatrix},
  \]
  where $c \in \RRR^m$. Then 
  \[
  \tilde E = E + \ell \ell'
  \]
  where $\ell\in \RRR^m$ is defined by $\ell_j = Q_0^{-1} (A\sigma_j, I)_\om$.
\end{proposition}
\begin{proof}
  Let $(\tau_{i},v_{i},q_i) = T e_i$. Then, by
  Lemma~\ref{lem:relation_lemma1}, $\TT \ee_i = 
  (\tau_i, v_i,q_i,\beta_i)$ with
  \[
  \beta_i = Q_0^{-1} (A\sigma_i, I)_\om.
  \]
  Hence,
  \begin{align*}
    \tilde E_{ij} 
    & = \bb\big( \ee_j, \TT\ee_i\big)
    \\
    & = \bb\big( \ee_j, (\tau_i, v_i,q_i,\beta_i)\,\big)
    \\
    & = \bb\big( \ee_j, (\tau_i, v_i,q_i,0)\,\big)
    + \bb\big( \ee_j, (0,0,0,\beta_i)\,\big)
    \\
    & = b(e_j,Te_i)
    + Q_0^{-1} (\beta_i I, A \sigma_j)_\om
    \\
    & = E_{ij} + \beta_i\beta_j.
  \end{align*}
  The result follows because $\beta_i=\ell_i$.
\end{proof}

A consequence of this proposition is that we can invert $\tilde E$
using the Sherman-Morrison formula~\cite{Hager89}, namely 
\begin{equation}
  \label{eq:11}
  (E + \ell \ell')^{-1} 
  =
  E^{-1}  - a (E^{-1} \ell) (E^{-1} \ell)',
  \quad\text{with}\quad
  a = \frac{1}{1 + \ell' E^{-1} \ell}.
\end{equation}
Therefore, to conclude this discussion, consider the matrix system
arising from the second DPG method~\eqref{eq:dpgapprx2}:
\begin{align*}
  \begin{bmatrix}
    \tilde E & c \\
    c' & d 
  \end{bmatrix}
  \begin{bmatrix}
    x \\ y 
  \end{bmatrix}
  = 
  \begin{bmatrix}
    g \\ 0 
  \end{bmatrix}.
\end{align*}
Since $\tilde{E}$ is nonsingular, the solution is given by 
\begin{align*}
y = -(d-c'\tilde{E}^{-1}c)^{-1}c'\tilde{E}^{-1}g,\qquad
x = \tilde{E}^{-1}g-\tilde{E}^{-1}cy.
\end{align*}
Hence the solution of the second DPG method can be obtained by solving
the two linear systems $\tilde E x_c= c$ and $\tilde E x_g = g$.  Each
of these systems can be solved using formula~\eqref{eq:11} at
essentially the same cost as solving a system involving~$E$.

\section{Error analysis}     \label{sec:anl}

In this section, we prove Theorems~\ref{thm:dpg1}
and~\ref{thm:dpg2}. We will give the proof in full detail for the more
difficult case of Theorem~\ref{thm:dpg2} first. Afterward, we will
indicate the minor modification required to prove
Theorem~\ref{thm:dpg1} in a similar fashion.  The plan is to use the
abstract DPG framework developed
in~\cite{DemkoGopal:DPGanl,DemkoGopal:2010:DPG1,DemkoGopal:2010:DPG2,DemkoGopalNiemi:2010:DPG3,ZitelMugaDemko:2010:DPG4},
summarized next.

\subsection{Abstract quasioptimality}    \label{ssec:quas}

Let $X$ with norm $\| \cdot\|_X$ be a reflexive Banach space over
$\RRR$, $Y$ with norm $ \| \cdot\|_Y$ be a Hilbert space over $\RRR$
with inner product $(\cdot,\cdot)_Y$, and $b(\cdot,\cdot) : X \times Y
\mapsto \RRR$ be a bilinear form.  The abstract trial-to-test operator
$T: X \mapsto Y$ is defined -- as before -- by $(Tu,v)_Y = b(u,v)$ for
all $u\in X$ and $y \in Y$.  
We can write the DPG method using an arbitrary subspace $X_h \subseteq
X$ even in this generality. Let $Y_h = T(X_h)$.

\begin{theorem}[see~{\cite[Theorem~2.1]{DemkoGopal:DPGanl}}]
  \label{thm:abs}
  Suppose $x \in X$ and $x_h \in X_h$ satisfy
  \begin{align*}
    b(x,y) & = l(y) && \forall \; v \in Y,
    \\
    b(x_h,y) & = l(y) && \forall \; y \in Y_h.
  \end{align*}
  Assume that 
  \begin{equation}
    \label{eq:inj}
    \{ w \in X :  \; b(w,v)  = 0,\; \forall v \in Y \} = \{ 0 \}
  \end{equation}
   and that there are positive constants $C_1, C_2$  such that 
   \begin{equation}
   \label{eq:equiv}
    C_1 \| y \|_Y  \le    \optnorm y Y    \le C_2 \| y \|_Y,
    \qquad \forall y \in Y,
  \end{equation}
  where the so-called ``optimal norm'' is defined by
  \begin{equation}
    \label{eq:opt}
    \optnorm{ y}{Y}  = \sup_{xw \in X } \frac{ b(x,y) }{ \| x \|_X }.
  \end{equation}
  Then 
  \[
  \| x - x_h \|_X \le \frac{C_2}{C_1} \inf_{z_h \in X_h} \| x - z_h \|_X.
  \]
\end{theorem}

\subsection{The optimal test norms of both methods}   \label{ssec:optimal}

We now see what is the norm defined by~\eqref{eq:opt} in the context
of the first and second DPG methods. From the structure of the
bilinear form it is easy to see that, for the first DPG method, the
optimal test norm is
\begin{align}
\label{opt_norm1}
\Vert(\tau,v,q)\Vert_{\text{opt},V}^2
& = \sup_{0\neq (\sigma,u,\hat{u},\hat{\sigma}_{n})\in U}
\dfrac{
b((\sigma,u,\hat{u},\hat{\sigma}_{n}),(\tau,v,q))^2}{\Vert(\sigma,u,\hat{u},\hat{\sigma}_{n})\Vert_{U}^2}
\\\nonumber
& 
=\Vert A\tau+\nabla v+q\Vert^{2}_{\Omega_{h}}+\Vert\dive\tau\Vert^{2}_{\Omega_{h}}
+\Vert[\tau\,n]\Vert_{\partial\Omega_{h}}^{2}
+\Vert[v n]\Vert_{\partial\Omega_{h}}^{2},
\end{align}
where  the ``jump'' terms are defined by 
\begin{subequations}
\begin{align}
\label{jump_norm_Hdiv}
\Vert[\tau\,n]\Vert_{\partial\Omega_{h}}& :=\sup_{0\neq u\in H^{1}_{0}(\Omega;\mathbb{V})}
\dfrac{\langle u,\tau\,n\rangle_{\partial\Omega_{h}}}{\Vert u\Vert_{H^{1}(\Omega)}}\\
\label{jump_norm_H1}
\Vert[v n]\Vert_{\partial\Omega_{h}}& :=\sup_{0\neq\sigma\in H(\text{div},\Omega;\mathbb{M})}
\dfrac{\langle v,\sigma\,n\rangle_{\partial\Omega_{h}}}{\Vert\sigma\Vert_{H(\text{div},\Omega)}}.
\end{align}
\end{subequations}
Similarly, for the second DPG method, taking the supremum over its
trial space, we have 
\begin{align}
\label{opt_norm2}
\optnn{(\tau,v,q,\beta)}^2
& 
=\Vert A\tau+\nabla v+q +Q_{0}^{-1}\beta AI \Vert^{2}_{\Omega_{h}}+\Vert\dive\tau\Vert^{2}_{\Omega_{h}}\\\nonumber
&\quad +\Vert[\tau\,n]\Vert_{\partial\Omega_{h}}^{2}
+\Vert[v n]\Vert_{\partial\Omega_{h}}^{2}
+\left| Q_{0}^{-1}\int_{\Omega}\trace (A\tau) \right|^{2}.
\end{align}
In either case, the optimal norms are inconvenient for practical
computations, due to the last two jump terms. These terms would make
the trial-to-test computation in~\eqref{eq:T} non-local.  Therefore, a
fundamental ingredient in our ensuing analysis is the
proof of equivalence of the optimal norm with the simpler ``broken''
norm (for the first DPG method)
\begin{equation}
  \label{eq:Vnorm}
  \| \, (\tau, v, q) \, \|_{V}^2 
  =  \| \tau \|_{\Hdiv\oh}^2 + \| v \|_{H^{1}(\oh)}^2+\Vert q\Vert_{\Omega}^{2}
\end{equation}
which does contain the jump terms. 
For the second DPG method, we will similarly prove
that~\eqref{opt_norm2} is equivalent to following norm
\begin{equation}
  \label{eq:Vnorm2}
  \| \, (\tau, v, q, \beta) \, \|_{V^{(2)}}^2 
  =  \| \tau \|_{\Hdiv\oh}^2 + \| v \|_{H^{1}(\oh)}^2
  +\Vert q\Vert_{\Omega}^{2}+\vert\beta\vert^{2}.
\end{equation}
These equivalences would verify condition~\eqref{eq:equiv}, so we
would be in a position to apply Theorem~\ref{thm:abs}.  Before we
proceed to prove these, let us verify the other
condition~\eqref{eq:inj}. We begin with a necessary preliminary.


\subsection{Korn inequalities}

We will need two well-known inequalities due to Korn. The {\em first}
Korn inequality asserts the existence of a constant $C>0$ such that
\begin{subequations}
  \begin{equation}
    \label{Korn_ineq1}
    \Vert v\Vert_{H^{1}(\om)}
    \leq C\Vert\veps(v)\Vert_{\om}\quad\forall v\in H_{0}^{1}(\om),
  \end{equation}
while the {\em second} Korn inequality gives a constant $C>0$ such that 
\begin{equation}
 \label{Korn_ineq2}
 \Vert v\Vert_{H^{1}(\om)}
 \leq C\left(\Vert v\Vert_{\om}+\Vert\veps(v)\Vert_{\om}\right)\quad
 \forall v\in H^{1}(\om).
\end{equation}
\end{subequations}
Above and in the remainder, we use $C$ to denote a generic constant
whose value, although possibly different at different occurrences,
will remain independent of the discrete approximation spaces.  These
inequalities can be found in many references. E.g.,
for~(\ref{Korn_ineq1}), see~\cite[eq.~(2.7)]{Horgan:1995:KornIneq},
and for~(\ref{Korn_ineq2}), see~\cite[Section 1.12 in Chapter
6]{MH:1993:MFE}.

\subsection{Uniqueness for the second DPG method}

In this subsection, we verify condition~\eqref{eq:inj} for the second DPG method.
 
\begin{lemma}
\label{lemma_uniqueness2}
With $U^{(2)}$ and $V^{(2)}$ as set in~\eqref{eq:dpg2}
suppose $(\sigma,u,\hat{u},\hat{\sigma}_{n},\alpha)\in U^{(2)}$ satisfies
\begin{equation}
\label{eq1_lemma_uniqueness2}
b^{(2)}((\sigma,u,\hat{u},\hat{\sigma}_{n},\alpha),(\tau,v,q,\beta))=
0,
\qquad\quad \forall (\tau,v,q,\beta)\in V^{(2)}.
\end{equation}
Then $(\sigma,u,\hat{u},\hat{\sigma}_{n},\alpha)=0$.
\end{lemma}

\begin{proof}
Equation (\ref{eq1_lemma_uniqueness2}) is the same as 
\begin{subequations}
\label{eq:14}
\begin{align}
\label{eq2_lemma_uniqueness2}
(A\sigma,\tau)_{K}+(u,\dive\tau)_{K}
-\langle\hat{u},\tau\,n\rangle_{\partial K} +Q_{0}^{-1}(\alpha AI,\tau)_{K}&=0
&& 
\forall\tau\in H(\text{div},K;\mathbb{S})\\
\label{eq3_lemma_uniqueness2}
(\sigma,\grad v)_{K}-\langle\hat{\sigma}_{n},v\rangle_{\partial K}
&=0
&&\forall v\in H^{1}(K;\mathbb{V})\\
\label{eq4_lemma_uniqueness2}
\int_{\Omega}\trace (A\sigma)=0.
\end{align}
\end{subequations}
Here, $\sigma\in L^{2}(\Omega_{h};\mathbb{S})$ because 
\begin{equation*}
(\sigma,q)_{\Omega}=0\quad\forall q\in L^{2}(\Omega_{h};\mathbb{K}).
\end{equation*}
We take $\tau\in \DD(K;\mathbb{S})$ and $v\in \DD (K;\mathbb{V})$
arbitrarily. Here, as usual, $\DD(D,\XXX)$ denotes the space of
infinitely smooth functions from $D$ into $\XXX$ that are compactly
supported in $D$.  Then we have
\begin{subequations}
\begin{align}
\label{eq5_lemma_uniqueness2}
A\sigma-\veps(u)+Q_{0}^{-1}\alpha AI=0\quad\text{ in }K\\
\label{eq6_lemma_uniqueness2}
\dive\sigma=0\quad\text{ in }K
\end{align}
\end{subequations}
in the sense of distributions, for every $K\in\Omega_{h}$. 
In particular, this implies that 
$u\in H^{1}(K;\mathbb{V})$ by ((\ref{eq5_lemma_uniqueness2}) and 
the second Korn inequality (\ref{Korn_ineq2})) 
and $\sigma\in H(\text{div},K;\mathbb{S})$ by (\ref{eq6_lemma_uniqueness2}).

Now we claim that we also have 
\begin{equation}
\label{eq7_lemma_uniqueness2}
\hat{\sigma}_{n}|_{\partial K}=\sigma\,n|_{\partial K}
\qquad\text{and}\qquad u|_{\partial K}=\hat{u}|_{\partial K}.
\end{equation}
The first identity in (\ref{eq7_lemma_uniqueness2}) is obtained by
integrating (\ref{eq3_lemma_uniqueness2}) by parts and using
(\ref{eq6_lemma_uniqueness2}) to find that $\langle
v,\hat{\sigma}_{n}-\sigma\,n\rangle _{1/2,\partial K}=0$ for every
$v\in H^{1}(\Omega_{h};\mathbb{V})$. Note that this identity 
implies that $\sigma\in H(\text{div},\Omega;\mathbb{S})$.  

To prove the second identity of~(\ref{eq7_lemma_uniqueness2}), we need
to proceed a bit differently.  Let $F$ be an arbitrary face of $\d K$,
say in three dimensions (the two dimensional case is similar and
simpler).  We will now show that given any $\eta \equiv (\eta_i) \in
\DD(F; \VVV)$, there is a $\tau \in H(\text{div},K;\SSS)$ such that
$\tau n \in L^2(\d K;\mathbb{V})$ is supported only on $F \subseteq \d
K$ such that $\tau n|_F = \eta$. To this end, we may, without loss of
generality, assume that $F$ is contained in the $xy$-plane, so that $n
= (0,0,1)'$. Then, set
\begin{equation}
\label{eq8_lemma_uniqueness2}
  \tau = \chi (x,y,z)
  \begin{bmatrix}
    0           & 0          & \eta_1(x,y) \\
    0           & 0          & \eta_2(x,y) \\
    \eta_1(x,y) & \eta_2(x,y) & \eta_3(x,y) 
 \end{bmatrix}
\end{equation}
where $\chi$ is an infinitely smooth cut-off function such that
(i)~the support of $\chi(x,y,0)$ is compactly contained in $F$ and
contains the support of $\eta$, and (ii)~$\chi \equiv 1$ on the
support of~$\eta$.  Clearly we can find such a cut-off function, and
furthermore, construct it so that $\tau n$ vanishes on all other faces
of $K$.

We use such $\tau$ to prove the second one in
(\ref{eq7_lemma_uniqueness2}).  First observe that
from~\eqref{eq1_lemma_uniqueness2}, we have
$(A\sigma,\tau)_{K}+(u,\dive\tau)_{K}
-\langle\hat{u},\tau\,n\rangle_{\partial K}+Q_0^{-1}(\alpha AI,\tau)_{K}=0$,
for all $\tau \in H(\text{div},K;\mathbb{S})$. Integrating by parts
(which is permissible since by~\eqref{eq:14}, $u\in H^1(K;\mathbb{V})$
and $\sigma \in \Hdiv{K;\SSS}$), and
using~\eqref{eq2_lemma_uniqueness2},
\begin{equation}
  \label{eq9_lemma_uniqueness2}
  \ip{ u - \hat u, \tau n}_{\d K} =0,\qquad
  \forall \tau \in H(\text{div},K;\SSS).
\end{equation}
Choosing $\tau$ as in~\eqref{eq8_lemma_uniqueness2}, this implies
\[
\int_F (u - \hat u) \;\eta \,\; ds= 0,\qquad
  \forall \eta \in \DD(F,\VVV).
\]
Hence $u|_F = \hat u|_F$ in $L^2(F)$ and this holds for all faces of
$\d K$. This proves the second identity
of~(\ref{eq7_lemma_uniqueness2}), which implies that $u\in
H^{1}_{0}(\Omega;\mathbb{V})$ (after also noting that
$\hat{u}|_{\partial\Omega}=0$).

Next, we choose $\tau=I$ on $\Omega$ and sum the terms
in~(\ref{eq2_lemma_uniqueness2}) over all $K\in\Omega_{h}$. Using the
fact that $u|_{\partial K}=\hat{u}|_{\partial K}$, we have that
\begin{equation*}
\int_{\Omega}\trace (A\sigma)+Q_{0}^{-1}\alpha\int_{\Omega}\trace (AI)=0.
\end{equation*} 
The first term vanishes due to~(\ref{eq4_lemma_uniqueness2}). Hence we
have shown that $\alpha=0$.

Now, choose $\tau=\sigma$ and $v=u$ in
(\ref{eq2_lemma_uniqueness2})--(\ref{eq3_lemma_uniqueness2}).  Summing up
these equations and canceling terms after integrating by parts, we
find that
\begin{equation}
\label{eq10_lemma_uniqueness2}
(A\sigma,\sigma)_{\Omega_{h}}+\langle u,\sigma\,n\rangle_{\partial\Omega_{h}}
-\langle\hat{u},\sigma\,n\rangle_{\partial\Omega_{h}}
-\langle u,\hat{\sigma}_{n}\rangle_{\partial\Omega_{h}}=0.
\end{equation}
Since $u\in H^{1}_{0}(\Omega;\mathbb{V})$ the last term on the left
had side vanishes. Furthermore, since we already showed that the
interelement jumps of $\sigma n$ are zero, the penultimate term
$\langle\hat{u},\sigma\,n\rangle_{\partial\Omega_{h}}$ also
vanishes. Note finally that $\ip{u,\sigma n}_{\partial\oh} =
\ip{u,\sigma n}_{\partial\om} =0$ as $u \in
H_0^1(\om)$. 
Thus, (\ref{eq10_lemma_uniqueness2}) 
implies that $\sigma=0$. 

Since $\sigma$ vanishes, by (\ref{eq5_lemma_uniqueness2}), we have that
$\veps (u)=0$. Since $u\in
H^{1}_{0}(\Omega;\mathbb{V})$, by the first Korn
inequality~\eqref{Korn_ineq1} we find that $u=0$. Since both $\sigma$ and $u$ vanish, by
(\ref{eq7_lemma_uniqueness2}), $\hat{u}$ and $\hat{\sigma}_{n}$ also
vanish.
\end{proof}

\subsection{An inf-sup condition}

In this subsection, we verify that the lower bound in the
condition~\eqref{eq:equiv} holds for the second DPG method. The lower
bound is the same as an inf-sup condition due to the definition of the
optimal norm. To prove this inf-sup condition, we use a modification
of the mixed method for linear elasticity with weakly imposed symmetry,
given in Appendix~\ref{sec:weakly-symm}.

\begin{lemma} \label{lem:lower_bound} 
  There is a positive constant $C_{1}$ such that for any
  $(\tau,v,q,\beta)\in V^{(2)}$,
\begin{equation*}
C_{1}\Vert (\tau,v,q,\beta)\Vert_{V^{(2)}}
\leq \Vert (\tau,v,q,\beta)\Vert_{\mathrm{opt},V^{(2)}}.
\end{equation*}
If, in addition, Assumption~\ref{asm:iso} holds, then the constant
$C_{1}$ can be chosen to be
\begin{equation}
\label{uniform_lower_bound}
C_{1}^{-1} = \bar{c}_{2}P_{0}^{-1}(\Vert A\Vert+B)^{2}B^{4}(\Vert A\Vert+P_{0}+1)^{2},
\end{equation}
where $\bar{c}_{2}$ is a positive constant independent of $A$.
\end{lemma}
\begin{proof}
  Given $(\tau,v,q,\beta)\in V^{(2)}$, we solve the mixed
  method~\eqref{eq:bal} with data $F_1 = \tau, F_2=-v, F_3 = q$ and
  $F_4 = N^{-1}\beta I/|\Omega|$, to get $(\sigma,u,\rho,a)\in
  H(\text{div},\Omega;\mathbb{M}) \times L^2(\Omega;\mathbb{V}) \times
  L^{2}(\Omega;\mathbb{K}) \times \mathbb{R}$.  According to
  Theorem~\ref{thm:balanced_dmf}, the component $u$ is in $H_0^1(\Omega;
  \VVV)$ and
  \begin{equation}
    \label{bound1_lower_bound}
    \Vert\sigma\Vert_{H(\text{div},\Omega)}+\Vert u\Vert_{H^{1}(\Omega)}+\Vert\rho\Vert_{\Omega}+\vert a\vert
    \leq C_{0} (\Vert\tau\Vert_{\Omega}+\Vert v\Vert_{\Omega}+\Vert q\Vert_{\Omega}+\vert \beta\vert).
  \end{equation}

  Within an element $K$, we can find the strong form of the equations
  in~\eqref{eq:bal} by choosing infinitely smooth test functions that
  are compactly  supported on $K$. We obtain
\begin{subequations}
\begin{align}
\label{eq5_lower_bound}
 A\sigma-\nabla u+\rho+aQ_{0}^{-1}AI & =\tau,\\
\label{eq6_lower_bound}
 \dive\sigma& =-v,\\
\label{eq7_lower_bound}
\skw\sigma& =q, \intertext{together with the last
  equation~\eqref{eq4_balanced_dmf} which can be restated simply as}
\label{eq8_lower_bound}
 Q_{0}^{-1}\int_{\Omega}\trace (A\sigma)& =\beta.
\end{align}
\end{subequations}
These, together with integration by parts, imply  that 
\begin{align*}
\Vert\tau\Vert_{\Omega}^{2}
& +\Vert v\Vert_{\Omega}^{2}+\Vert q\Vert_{\Omega}^{2} +\vert\beta\vert^{2}\\
&=(A\sigma-\nabla u+\rho+aQ_{0}^{-1}AI,\tau)_{\Omega}
-(\dive\sigma,v)_{\Omega}+(\skw\sigma,q)_{\Omega}+\beta Q_{0}^{-1}\int_{\Omega}\trace (A\sigma)
\\
&=(\sigma,A\tau)_{\Omega_{h}}+(u,\dive\tau)_{\Omega_{h}}-\langle u,\tau n\rangle_{\partial\Omega_{h}}
+(\rho,\tau)_{\Omega}+aQ_{0}^{-1}\int_{\Omega}\trace (A\tau)
\\
&\qquad\qquad +(\sigma,\nabla v)_{\Omega_{h}}-\langle v,\sigma n\rangle_{\partial\Omega_{h}}
+(\skw\sigma,q)_{\Omega}+(\sigma,\beta Q_{0}^{-1}AI)_{\Omega_{h}}
\\
 &=(\sigma,A\tau)_{\Omega_{h}}+(u,\dive\tau)_{\Omega_{h}}-\langle u,\tau n\rangle_{\partial\Omega_{h}}
 +aQ_{0}^{-1}\int_{\Omega}\trace (A\tau)
 \\
&\qquad\qquad +(\sigma,\nabla v)_{\Omega_{h}}-\langle v,\sigma n\rangle_{\partial\Omega_{h}}
 +(\sigma,q)_{\Omega}+(\sigma,\beta Q_{0}^{-1}AI)_{\Omega_{h}}
\intertext{since $\tau$ is symmetric. Rearranging and applying Cauchy-Schwarz inequality,}
\Vert\tau\Vert_{\Omega}^{2}
& +\Vert v\Vert_{\Omega}^{2}+\Vert q\Vert_{\Omega}^{2} +\vert\beta\vert^{2}
\\
& =(\sigma,A\tau+\nabla v+q+\beta Q_{0}^{-1}AI)_{\Omega_h}+(u,\dive\tau)_{\Omega_{h}}
+aQ_{0}^{-1}\int_{\Omega}\trace (A\tau)\\
&\qquad\qquad -\langle u,\tau n\rangle_{\partial\Omega_{h}}-\langle v,\sigma n\rangle_{\partial\Omega_{h}}
\\
&\le
\Vert\sigma\Vert_{\Omega}\Vert A\tau+\nabla v+q+\beta Q_{0}^{-1}AI\Vert_{\Omega_{h}}
+\Vert u\Vert_{\Omega}\Vert\dive\tau\Vert_{\Omega_{h}}+\vert a\vert\cdot\vert Q_{0}^{-1}\int_{\Omega}\trace (A\tau)\vert\\
&\qquad\qquad +\Vert[\tau\,n]\Vert_{\partial\Omega_{h}}\Vert u\Vert_{H^{1}(\Omega)} 
+\Vert[v n]\Vert_{\partial\Omega_{h}}\Vert\sigma\Vert_{H(\text{div},\Omega)}
\\
&\leq
2 \optnn{ (\tau,v,q,\beta)}
\big(\Vert \sigma\Vert_{H(\text{div},\Omega)}^{2}
+\Vert u\Vert_{H^{1}(\Omega)}^{2}+\Vert \rho\Vert_{\Omega}^{2}+\vert a\vert^{2}\big)^{1/2}
\end{align*}
By (\ref{bound1_lower_bound}), we have that
\begin{equation}
\label{bound2_lower_bound}
\Vert\tau\Vert_{\Omega}^{2}+\Vert v\Vert_{\Omega}^{2}+\Vert q\Vert_{\Omega}^{2}+\vert\beta\vert^{2}
\leq 4C_{0}^2 \optnn{ (\tau,v,q,\beta) }^2.
\end{equation}
Furthermore, 
since 
$\Vert A\tau+\nabla v+q+\beta Q_{0}^{-1}AI\Vert_{\Omega_{h}}\leq
\optnn{  (\tau,v,q,\beta)}$, by triangle inequality,
\[
\| \nabla v \|_{\Omega_h}
\le 
\| A \| \|\tau \|_\om + \| q\|_\om + \| A  I\|Q_0^{-1} | \beta|
 + 
\optnn{ (\tau,v,q,\beta) }.
\]
which implies that 
\begin{equation*}
\Vert \nabla v\Vert_{\Omega_{h}}\leq c_1 \Vert (\tau,v,q,\beta)\Vert_{opt,V^{(2)}}.
\end{equation*}
for a positive constant $c_1$.

If, in addition the material is isotropic in the sense of
Assumption~\ref{asm:iso}, then using the notations of the
assumption, the constant $c_1$ can be chosen to be
\begin{equation*}
c_{1} = 2(\Vert A\Vert+B)C_{0}.
\end{equation*}

Finally, since $\Vert\dive\tau\Vert_{\Omega_{h}}\leq \optnn{
  (\tau,v,q,\beta)}$, we can control all terms in forming the norm
$\Vert\dive\tau\Vert_{\Omega_{h}}$, i.e.,
\begin{equation*}
\Vert (\tau,v,q,\beta)\Vert_{V^{(2)}}\leq c_{2}(2C_{0}+c_{1})\Vert (\tau,v,q,\beta)\Vert_{opt,V^{(2)}}
\end{equation*}
with a constant $c_{2}$ is independent of $A$. The lemma follows with
$C_{1}^{-1}=c_{2}(2C_{0}+c_{1})$. 
In the case of isotropic material, observe that 
\begin{align*}
& C_{1}^{-1}=c_{2}(2C_{0}+c_{1})=2c_{2}(\Vert A\Vert+B+1)C_{0}\leq 4c_{2}(\Vert A\Vert+B)C_{0}\\
&\quad \leq 4c_{2}(\Vert A\Vert+B)\bar{c}_{1}P_{0}^{-1}B^{4}(\Vert A\Vert+P_{0}+1)^{2}(\Vert A\Vert+B)\\
&\quad\leq\bar{c}_{2}P_{0}^{-1}(\Vert A\Vert+B)^{2}B^{4}(\Vert A\Vert+P_{0}+1)^{2},
\end{align*}
where $\bar{c}_{2}$ is a positive constant independent of $A$.
\end{proof}

\subsection{Upper bound} 

Now we show that the upper inequality of condition~\eqref{eq:equiv} can be verified for the second DPG method.

\begin{lemma}
\label{upper_bound}
There is a positive constant $C_{2}$ such that for any $(\tau,v,q,\beta)\in V^{(2)}$,
\begin{equation*}
\optnn{ (\tau,v,q,\beta)}
\leq \bar{c}_{3}(\Vert A\Vert + B)\Vert (\tau,v,q,\beta)\Vert_{V^{(2)}}.
\end{equation*}
Here, $\bar{c}_{3}$ is a positive constant independent of $A$.
\end{lemma}

\begin{proof}
  Let us first prove an upper bound for the jump terms. Integrating by
  parts locally and applying Cauchy-Schwarz inequality,
\[
\begin{aligned}
\jmpn{ \tau\,n}
& = \sup_{ w \in H_0^1(\om;\mathbb{V})}   \frac{ \ip{ w, \tau \,n }_{\partial\oh} }
{ \; \| w \|_{H^{1}(\om)} }
= 
 \sup_{ w \in H_0^1(\om;\mathbb{V})}   \frac{ (\grad w,\tau)_{\oh} +  (w, \dive \tau)_\oh}
{ \; \| w \|_{H^{1}(\om)} }
\\
 &\le \| \tau \|_{\Hdiv\oh}.
\end{aligned}
\]
We use a similar argument for the other jump, i.e.,
\[
\begin{aligned}
\jmpn{ v n}
& = \sup_{ \varsigma \in \Hdiv{\om;\MMM} }
\frac{ \!\! \ip{ v, \varsigma \,n }_{\partial\oh} }
{ \;\;\; \| \varsigma \|_{\Hdiv{\om;\MMM}} }
= 
\sup_{ \varsigma \in \Hdiv{\om;\MMM} }
\frac{ (\grad v,\varsigma)_{\oh} +  (v, \dive \varsigma)_\oh}
{ \;\;\; \| \varsigma \|_{\Hdiv{\om;\MMM}} }
\\
& 
\le \| v \|_{H^1(\oh)}.
\end{aligned}
\]
The remainder of the proof is straightforward.
\end{proof}

\subsection{Proof of Theorem~\ref{thm:dpg1}}

We apply the abstract result of Theorem~\ref{thm:abs}.
Assumption~\eqref{eq:inj} is verified by
Lemma~\ref{lemma_uniqueness2}.  The lower inequality
of~\eqref{eq:equiv} is verified by Lemma~\ref{lem:lower_bound}, and
the upper inequality is verified by
Lemma~\ref{upper_bound}. \hfill$\Box$

\subsection{Proof of Theorem~\ref{thm:dpg2}}

The analysis of the first DPG method is in many ways simpler than the
above detailed analysis of the second DPG method. Just as in the proof
of Theorem~\ref{thm:dpg1}, we only need to verify the
conditions~\eqref{eq:inj} and~\eqref{eq:equiv} of the abstract result.

The proof of the uniqueness condition~\eqref{eq:inj} is similar and
simpler than the proof of Lemma~\ref{lemma_uniqueness2}, so we omit
it.

The proof of the upper inequality in condition~\eqref{eq:equiv} is the
same the proof of Lemma~\ref{upper_bound}.

The proof of the lower inequality in condition~\eqref{eq:equiv} for
the first DPG method is analogous and simpler than the proof of
Lemma~\ref{lem:lower_bound}. To highlight the main difference, instead
of considering the mixed method in Appendix~\ref{sec:weakly-symm}, we
now need only use the standard mixed method with weakly imposed stress
symmetry. In other words, the analogue of~\eqref{bound1_lower_bound}
is now obtained as follows: Given $(\tau,v,q)\in V$, we solve the
following variational problem to find $(\sigma,u,\rho)\in
H(\text{div},\Omega;\mathbb{M}) \times L^{2}(\Omega;\mathbb{V}) \times
L^{2}(\Omega;\mathbb{K})$ satisfying
\begin{subequations}
\begin{align}
\label{eq1_lower_bound_simple}
 (A\sigma,\delta\tau)_{\Omega}+(u,\dive\delta\tau)_{\Omega}+(\rho,\delta\tau)_{\Omega}
& =(\tau,\delta\tau)_{\Omega} &&
 \forall\delta\tau\in H(\text{div},\Omega;\mathbb{M}),
\\
\label{eq2_lower_bound_simple}
(\dive\sigma,\delta v)_{\Omega}
& =-(v,\delta v)_{\Omega} && \forall\delta v\in L^{2}(\Omega;\mathbb{V}),
\\
\label{eq3_lower_bound_simple}
(\sigma,\delta q)_{\Omega}
& =(q,\delta q)_{\Omega} &&  \forall \delta q\in L^{2}(\Omega;\mathbb{K}).
\end{align}
\end{subequations}
Then we use the standard stability estimate~\cite{ArnolFalkWinth07}
for this method to get the analogue of~\eqref{bound1_lower_bound} and
proceed as in the proof of Lemma~\ref{lem:lower_bound}.\hfill$\Box$

\section{Examples of trial spaces and convergence rates}   \label{sec:rates}

The trial subspaces of both the first and the second DPG methods
(namely $U_h$ and $U_{h}^{(2)}$) were unspecified in
Theorem~\ref{thm:dpg1} and theorem~\ref{thm:dpg2}.  This section is
devoted to two examples of trial spaces and how one can use
Theorems~\ref{thm:dpg1} and~\ref{thm:dpg2} to predict $h$ and
$p$~convergence rates for these examples.  The examples we have in
mind are DG spaces built on a tetrahedral mesh and a cubic mesh.  We
only consider the case of the first DPG method (as the same
convergence rates can be derived analogously for the second DPG
method).

If $D$ is a simplex, let $P_p(D)$ denote the set of functions that are
restrictions of (multivariate) polynomials of degree at most $p$ on a
domain $D$. If $D$ is cubic, then we write it as a tensor product of
three intervals $D = D_x \otimes D_y \otimes D_z$ and define
$Q^{p,q,r}(D) = P_{p}(D_x)\otimes P_{q}(D_y)\otimes P_{r}(D_z)$ which
is the space of polynomials of degree at most $p,q,r$ with respect to
$x,y,z$, resp. As with Sobolev spaces, when these notations may also
be augmented with a range vector space, i.e., $P_p(D;\mathbb{S})$
denotes the space of symmetric matrix-valued functions whose
components are polynomials of degree at most~$p$, etc.


Recall that -- see~\eqref{eq:Uh} -- to specify $U_h$, we must specify its
four component spaces. If $\Omega_{h}$ is a tetrahedral mesh, we set
\begin{subequations} \label{eq:spaces}
\begin{align}
  \Sigma_{h,p} & = \{ \rho: \rho|_K \in P_p(K;\mathbb{S}) \}, &
  W_{h,p} & = \{ v : v|_K  \in P_p(K;\mathbb{V}) \},
\intertext{while if $\Omega_{h}$ is a cubic mesh, then we set}
 \Sigma_{h,p} & = \{ \rho: \rho|_K \in Q^{p,p,p}(K;\mathbb{S}) \}, &
  W_{h,p} & = \{ v : v|_K  \in Q^{p,p,p}(K;\mathbb{V}) \}. &
\end{align}
Note that we have chosen the subspaces to consist of {\em symmetric}
matrix polynomials. This is clearly allowed since the only requirement
for the discrete trial subspace was that $U_h \subset U$. (The
corresponding automatically generated test space ensures stability of
the resulting method. We emphasize that this stabilization mechanism
is different from mixed methods.)  The numerical flux space is set as
follows:
\begin{align}
F_{h,p}
  & =  \{ \eta : \;\eta|_E \in P_p(E;\mathbb{V}) \; \forall \text{ mesh faces }E \}
  && \text{ if }\Omega_{h}\text{ is a tetrahedral mesh,}
\\
F_{h,p}
  & =  \{ \eta : \;\eta|_E \in Q^{p,p}(E;\mathbb{V}) \; \forall \text{ mesh faces }E \}
  && \text{ if }\Omega_{h}\text{ is a cubic mesh.}  
\end{align}
In either case, we define the numerical trace space by
\begin{align}
M_{h,p+1}
  & = \{ \eta: \exists\, w \in W_{h,p+1} \cap H_0^1(\om;\mathbb{V}) \text{ such that } 
  \eta|_{\d K} = w|_{\d K} 
  \; \forall K \in \oh \}.
\end{align}
\end{subequations}
Since $p\ge 0$, the space $M_{h,p+1}$ is non-trivial.

Let us apply Theorem~\ref{thm:dpg2} with these as trial spaces for
each solution component.  Then, if we know how the best approximation
error converges in terms of $h$ and $p$, we can conclude rates of
convergence. It is well known that for $s>0$,
\begin{equation}
  \label{eq:34}
  \inf_{w_h \in W_{h,p} } 
  \| u - w_h \|_\om \le C h^s\pt^{-s} | u |_{H^s(\om)},
  \qquad (s \le  p+1).
\end{equation}
Here $\pt = \max(p,2)$.
A similar best approximation estimate obviously holds for $\sigma$ as
well.  Note that since the exact stress $\sigma$ is symmetric, it can
be approximated to optimal accuracy by the symmetric subspace
$\Sigma_{h,p}$.

Next, consider the flux and trace best approximations in the quotient
topology defined by~\eqref{trace_norm} and~\eqref{flux_norm}.  Since
the exact trace $\hat u$ is the trace of the exact solution~$u$, and
since the exact flux $\sgn$ is the trace of the normal components
of~$\sigma$ along each interface of $\partial\Omega_{h}$, we have
\begin{align*}
  \inf_{\hat z_h \in M_{h,p}} 
  \| \hat u - \hat z_h \|_{H_0^{1/2}(\partial\oh) }
  & \le  \| u - \gpi u \|_{H^1(\om)},
  \\
  \inf_{\etah \in Q_{h,p}} 
  \| \sgn - \etah \|_{H^{-1/2}(\partial\oh) }
  & \le  \| \sigma - \dpi \sigma \|_{\Hdiv\om},
\end{align*}
where $\gpi u \in H_0^1(\om;\mathbb{V})$ and $\dpi \sigma \in
H(\text{div},\om;\mathbb{M})$ are suitable projections, such that
their traces $\gpi u|_E$ and $\dpi\sigma\,n |_E$ on any mesh face $E$
is in $P_p(E;\mathbb{V})$ or $Q^{p,p}(E;\mathbb{V})$.  These
conforming projectors providing approximation estimates with constants
independent of~$p$ are available from recent works
in~\cite{Demko:2008:PBI,DemkoBuffa:2005:PBI,DemkoGopalSchob:2008:EOP1,DemkoGopalSchob:2009:EOP2,DemkoGopalSchob:2009:EOP3}. 
Specifically, \cite[Theorem~8.1]{DemkoGopalSchob:2009:EOP3}
gives
\begin{subequations}
  \label{eq:33}
 \begin{align}
    \| u - \gpi u \|_{H^1(\om)}
    & \le 
    C \ln(\pt)^2 \, h^s \pt^{-s} 
    | u |_{H^{s+1}(\om)},
    \qquad (s\le p),
    \\
    \label{eq:33-div}
    \| \sigma - \dpi \sigma \|_{L^2(\om)}
    & \le 
    C \ln(\pt)\, h^s \pt^{-s}
    |\sigma |_{H^{s+1}(\om)},
    \qquad (s\le p+1).
 \end{align}
\end{subequations}
whenever $ s>1/2$ for tetrahedral meshes.  The same results for cubic
meshes are available from~\cite{Demko:2008:PBI}.  Note that we have
chosen $\dpi$ to be a projector into the Raviart-Thomas space.  The
projector into the Raviart-Thomas space satisfies $\dive \dpi \sigma =
\varPi_p \dive\sigma$ (where $\varPi_p$ denote the $L^2$-orthogonal
projection into $W_{h,p}$).  Hence,
\begin{equation}
  \label{eq:38}
\| \dive ( \sigma - \dpi \sigma ) \|_{L^2(\om)}
\le C   h^s  \pt^{-s}
    |\dive  \sigma |_{H^s(\om)},
    \qquad (s\le p+1).
\end{equation}
Finally, comparing the rates of convergence in~\eqref{eq:34},
\eqref{eq:33} and~\eqref{eq:38}, we find that to obtain a full
$O(h^{p+1})$ order of convergence, we must increase the polynomial
degree of the numerical trace space to $p+1$. Combining these
observations, we have the following corollary.

\begin{corollary}[$h$ and $p$ convergence rates]
  \label{cor:rates}
  Let $\oh$ be a shape regular mesh (either tetrahedral or cubic) and
  let $h$ denote the maximum of the diameters of its elements.  Let
  $\DE$ and $\DE^{(2)}$ denote the (previously defined) discretization
  errors of the first and second DPG methods, resp.  Using the spaces
  defined in~\eqref{eq:spaces}, set
  \[
  U_h = \Sigma_{h,p} \times W_{h,p} \times M_{h,p+1} \times F_{h,p}
  \]
  for the first DPG method and 
  \[
  U_h^{(2)} = \Sigma_{h,p} \times W_{h,p} \times M_{h,p+1} \times F_{h,p}\times \mathbb{R}
  \]
  for the second DPG method. Then
  \begin{align*}
    \DE
    & \le \,
    C_{I\;}  \ln(\pt)^2 \,h^{s} \pt^{-s} ( \|\sigma \|_{H^{s+1}(\om)} + \| u \|_{H^{s+1}(\om)})    
    \\
    \DE^{(2)}
    & \le \,
    C_{II}  \ln(\pt)^2 \,h^{s} \pt^{-s} ( \|\sigma \|_{H^{s+1}(\om)} + \| u \|_{H^{s+1}(\om)})
  \end{align*}
  for all $1/2 <s\le p+1$. The constants $C_I$ and $C_{II}$ are
  independent of $h$ and $p$, but dependent on the shape regularity
  and $A$.  If Assumption~\ref{asm:iso} (isotropy) holds, then
  $C_{II}$ is independent of $Q_{0}$, so the second estimate does not
  degenerate as the Poisson ratio goes to 0.5.
\end{corollary}

In the same way, one can derive convergence rates for other element
shapes and spaces from Theorems~\ref{thm:dpg1} and
Theorem~\ref{thm:dpg2}.

\begin{remark}[Symmetric and conforming stresses]
  If an application demands stress approximations $\sigma_h$ that
  are both symmetric and $\dive$-conforming (i.e., if one needs
  $\sigma_h$ to be in the space $\Hdiv{\om;\SSS}$ defined
  in~\eqref{eq:18}), then the DPG method can certainly give such
  approximations. We only need to choose a trial subspace
  \begin{equation}
    \label{eq:1}
      \Sigma_{h,p} \subset    \Hdiv{\om;\SSS}
  \end{equation}
  instead of the choice $\Sigma_{h,p} \subset L^2(\om, \SSS)$ we made
  in~\eqref{eq:spaces} above. Obviously $ \Hdiv{\om;\SSS} \subset
  L^2(\om, \MMM)$, so~\eqref{eq:1}, together with the other component
  spaces as set previously, would result in a trial space $U_h$ that
  is conforming in our ultraweak variational framework.  Hence,
  Theorems~\ref{thm:dpg1} and~\ref{thm:dpg2} continue to apply. Notice
  that stability of the resulting DPG method is ensured even with this
  choice because the method adapts its test space to any given trial
  subspace. The first example of $\Sigma_{h,p}$
  satisfying~\eqref{eq:1} that comes to mind is the finite element
  of~\cite{ArnolAwanoWinth08}. However their space is too rich because
  they had to ensure a discrete inf-sup condition. Since we have
  separated out the stability issue, we have other simpler and
  inexpensive options. E.g., we may choose $\Sigma_{h,p}$ to consist
  of symmetric matrix functions, each of whose entries are in $L_{h,p}
  = \{ \rho \in H^1(\om,\SSS): \rho|_{K} \in P_p(K,\SSS) \}$ (i.e.,
  continuous Lagrange finite element functions). The locality of our
  test space construction is not destroyed with this choice. Moreover,
  $p$-optimal interpolation estimates are known for this space, so we
  can proceed as above to state an analogue of
  Corollary~\ref{cor:rates}. Of course, the same remarks also apply
  for displacement approximations, e.g., we may choose
  $H^1$-conforming subspaces to approximate the displacement.
\end{remark}

\section{Numerical results}   \label{sec:numerical}

In this section, we present numerical results for the first DPG method
using two test cases: a smooth solution on a square domain and a
singular solution on an L-shaped domain. All numerical experiments
were conducted using a pre-existing $hp$-adaptive finite
element package~\cite{hpbook}.

\subsection{Discrete spaces}

Following \cite{DemkoGopal:DPGanl}, we considered a 2D domain $\Omega$
divided into conforming or 1-irregular quadrilateral meshes. Let
$\Omega_h$ denote the collection of mesh elements and $\mathcal{E}_h$
denote the collection of mesh edges. A polynomial degree $p_K \geq 1$
is assigned to each element and a degree $p_E$ is assigned to each
mesh edge $E$. For the first method, the practical trial space is $U_h
= \Sigma_h \times W_h \times M_h \times F_h$ where
\begin{subequations}
\begin{align}
K_h & = \left\{ v : v|_K \in Q_{p_K,p_K}(K), \forall K \in \Omega_h \right\} \\
\label{eq:Sigma}
\Sigma_h & = (K_h)^3  \\
W_h & = (K_h)^2  \\
M_h & = \left\{ \mu : \mu|_E \in P_{p_E+1}(E), \forall E \in \mathcal{E}_h\; \mbox{and} \;\mu \in H_0^1(\Omega)\right\} \\
F_h & = \left\{ \eta : \eta|_E \in P_{p_E}(E), \forall E \in \mathcal{E}_h\right\}
\end{align}  
\end{subequations}
(and, as before, $Q_{l,m}$ is the space of bivariate polynomials which are of
degree at most $l$ in $x$ and $m$ in $y$). Notice that
in~\eqref{eq:Sigma}, we interpret each element of~$(K_h)^3$ as a
symmetric matrix with entries in $K_h$ (so the stress approximations
are strongly symmetric).
The edge order $p_E$ is determined using the maximum
rule, i.e. $p_E$ is the maximum order of all elements adjacent to edge
$E$.

Recall that the test space is determined by the trial-to-test operator
$T:U \mapsto V$, which in turn requires inverting the Riesz map
corresponding to the inner product in the test space. In practice,
this is solved on the discrete level by using $\tilde{T}: U \mapsto
\tilde{V}$ which is defined as
\begin{align}
(\tilde{T}u,\tilde{v})_V = b(u,\tilde{v})\, &  \hspace{0.6in} \forall \; \tilde{v} \in \tilde{V}
\end{align}
where $\tilde{V}$ is a finite dimensional subspace of $V$. For our
implementation, $\tilde{V}$ is defined as
\begin{align}
\tilde{V} &= \left\{(\tau,v) : \tau|_K \in (Q_{\tilde{p}_K,\tilde{p}_K})^3 \; \mbox{and} \; v|_K \in (Q_{\tilde{p}_K,\tilde{p}_K})^2 \right\}
\end{align}
where $\tilde{p}_K = p_K + \delta p$. The default choice for the enrichment degree $\delta p$ is 2. Numerical experience shows that this is a sufficient choice for most problems (see figure \ref{fig:approx}). 

Finally, we approximate the energy norm of the error using the error representation function $\tilde{e} \in \tilde{V}$ where $\tilde{e} = \tilde{T}(u-u_h)$. Note that by the definition of $\tilde{T}$, the error representation function can be computed element wise by solving the variational problem
\begin{align}
(\tilde{e},\tilde{v})_V = (\tilde{T}(u-u_h),\tilde{v})_V = b(u-u_h,\tilde{v}) = l(\tilde{v}) - b(u_h,\tilde{v}).
\end{align}
This implies that the energy norm of the error is approximated by
\begin{align}
\Vert u-u_h \Vert_E & = \sup_{v \in V} \frac{\vert b(u-u_h,v) \vert}{\Vert v \Vert_V} 
 = \Vert T(u-u_h) \Vert_V 
 \approx \Vert \tilde{T}(u-u_h) \Vert_V 
 = \Vert \tilde{e} \Vert_V.
\end{align}
From Theorem \ref{thm:abs}, it is known that the energy error is
equivalent to the standard norm error on $U$, so our choice of error
indicator is justified assuming that the approximation of $T$ by
$\tilde{T}$ is sufficient.

In all cases, the standard test space norm is used for the inversion of the Riesz map, i.e.,
\begin{align}
\label{eq:vvv}
\Vert (\tau,v) \Vert_V^2 &= \Vert \tau \Vert^2_{\Hdiv\oh} + \Vert v \Vert^2_{H^1(\Omega_h)}.
\end{align}
Note that due to the strong symmetry of functions in the discrete
space $\Sigma_h$, the term involving $q$
in~\eqref{binear_form_concrete} vanishes. Hence, at the discrete
level, we may (omit all $q$'s and) work with the reduced test space
$V=H(\text{div},\Omega_{h};\mathbb{S})\times
H^{1}(\Omega_{h};\mathbb{V})$ instead of~\eqref{test_space}. This is
why we use the norm~\eqref{eq:vvv} in place of~\eqref{test_norm}.

\subsection{Test Problems}

\subsubsection{Smooth Solution}
The first problem studied was a smooth solution with manufactured body force found by applying the elasticity equations to the exact displacements
\begin{align}
u_x &= \sin (\pi x) \sin (\pi y) \\
u_y &= \sin (\pi x) \sin( \pi y) 
\end{align}
over a unit square domain $\Omega = (0,1) \times (0,1)$ with $u_x, u_y$ prescribed on $\partial \Omega$. For all cases, $\Omega_h$ is initially a uniform mesh of 4 square elements.

\subsubsection{L-Shaped Steel}

The second problem considered was the classical L-shaped domain
problem with the material properties of steel. In polar coordinates
$r,\theta$ around the re-entrant corner, the singular solution takes
the form (see e.g.,~\cite{Vasil88} or~\cite[\S~4.2]{Grisv92})
\begin{align}
\sigma_{r}  &  =r^{a-1}\left[  F^{^{\prime\prime}}\left(  \theta\right)
+\left(  a+1\right)  F\left(  \theta\right)  \right]  \\
\sigma_{\theta}  &  =a\left(  a+1\right)  r^{a-1}F\left(  \theta\right)  \\
\sigma_{r\theta}  &  =-ar^{a-1}F^{^{\prime}}\left(  \theta\right) \\
u_{r}  &  =\frac{1}{2\mu}r^{a}\left[  -\left(  a+1\right)  F\left(
\theta\right)  +\left(  1-\frac{\nu}{1+\nu}\right)  G^{^{\prime}}\left(
\theta\right)  \right]  \\
u_{\theta}  &  =\frac{1}{2\mu}r^{a}\left[  -F^{^{\prime}}\left(
\theta\right)  +\left(  1-\frac{\nu}{1+\nu}\right)  \left(  a-1\right)
G\left(  \theta\right)  \right] 
\end{align}
where $\nu = \frac{\lambda}{2(\lambda+\mu)}$ is Poisson's ratio and
the functions $F(\theta)$ and $G(\theta)$ are given by
\begin{align}
F\left(  \theta\right) & =C_{1}\sin\left(  a+1\right)  \theta+C_{2}\cos\left(
a+1\right)  \theta+C_{3}\sin\left(  a-1\right)  \theta+C_{4}\cos\left(
a-1\right)  \theta,\\
G\left(  \theta\right) & =\frac{4}{a-1}\left[  -C_{3}\cos\left(  a-1\right)
\theta+C_{4}\sin\left(  a-1\right)  \theta\right].
\end{align}
To determine the constants $C_1,C_2,C_3,C_4$ and $a$, we use the
kinematic boundary conditions along the edges forming the reentrant
corner (which without loss of generality we can take to be the edges
$\theta = \pm 3\pi/4$). We can obtain a square integrable solution
that satisfies $\dive \sigma=0$ by setting 
$C_2 = C_4 = 0$, $C_3=1$ and 
\begin{align}
C_{1}  &  =\frac{\left[  4\left(  1-\frac{\nu}{1+\nu}\right)
-\left(  a+1\right)  \right]  
\sin\left(  (a-1) 
 \frac{3}{4}\pi \right) }{\left(
a+1\right)  \sin\left(  (a+1)  \frac{3}{4}\pi \right)}
\end{align}
and letting $0<a<1$ be the solution of the transcendental equation
\begin{align*}
  C_1\,\cos\left( \frac { 3\left( a+1 \right) \pi} {4}  \right)  
  \left( a +1 \right)
  & + 
  \cos \left(\frac {3  \left( a-1 \right) \pi}{4}  \right)  \left( 
    a-1 \right) 
\\
& + 4\, \left( 1-{\frac {\nu}{1+\nu}} \right) \cos \left( 
\frac {3  \left( a-1 \right) \pi}{4}  \right) =0.
\end{align*}
Numerically solving for $a$ with the material properties of
steel~\cite{FungTong01}, namely $\lambda = 123$ GPa, $\mu = 79.3$ GPa,
we obtain $ a \approx 0.6038.  $ This implies that all stress
components have a singularity of strength (approximately)
$r^{-0.3962}$ while the displacement components are smooth at the
origin (but have singular derivatives). We can thus expect the stress
components to be in (a space close to) $H^{0.6038-\epsilon}$ and
the displacement components in $H^{1.6038-\epsilon}$ for
$\epsilon>0$.

\subsection{Convergence rates}

\begin{figure}
\begin{center}
  \subfigure[Uniform $h$-refinements.]{
    \includegraphics[width=0.47\textwidth,angle=0]{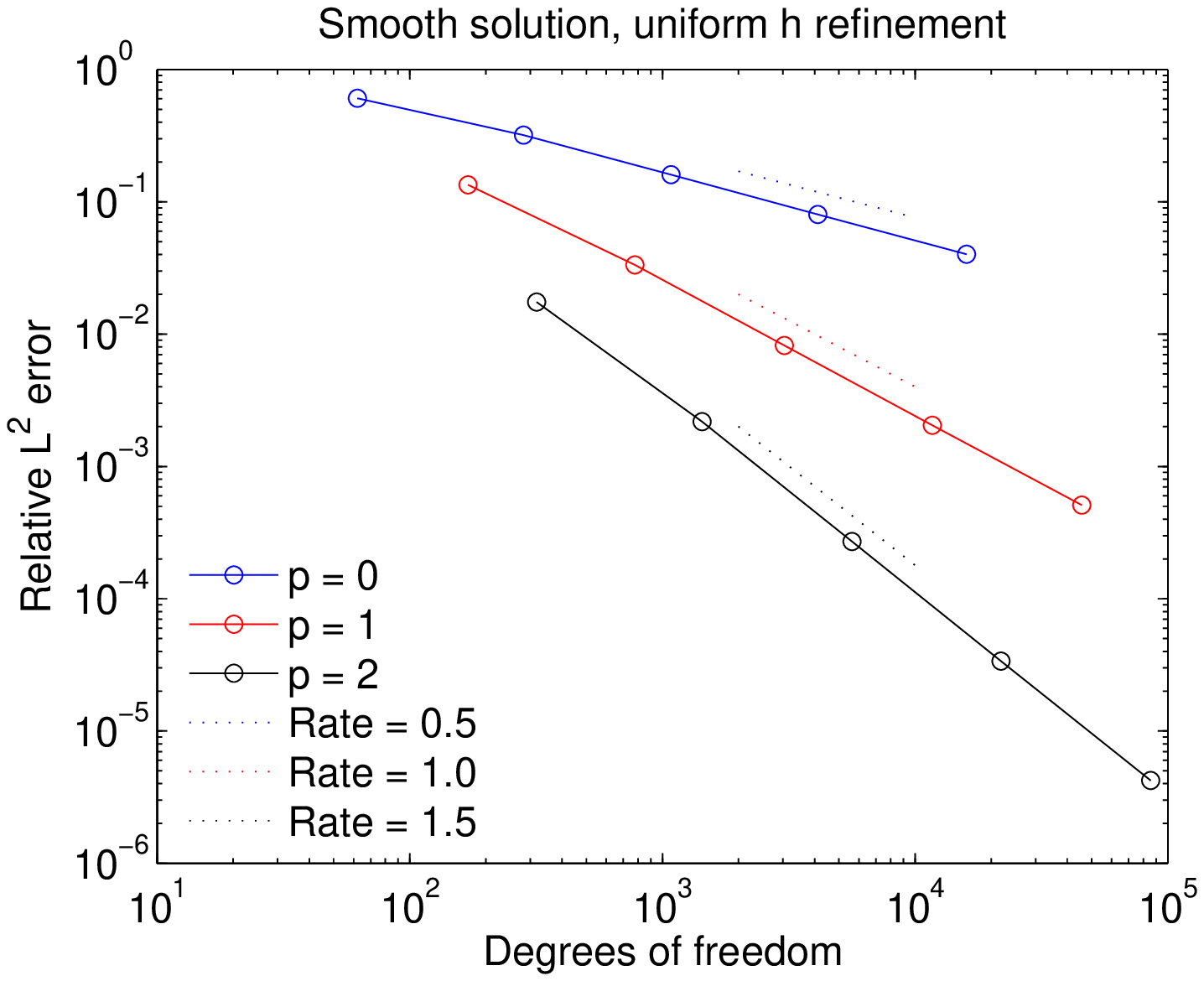}
    \label{fig:urefsmooth:h}
  }
  \subfigure[Uniform $p$-refinements.]{
    \includegraphics[width=0.47\textwidth,angle=0]{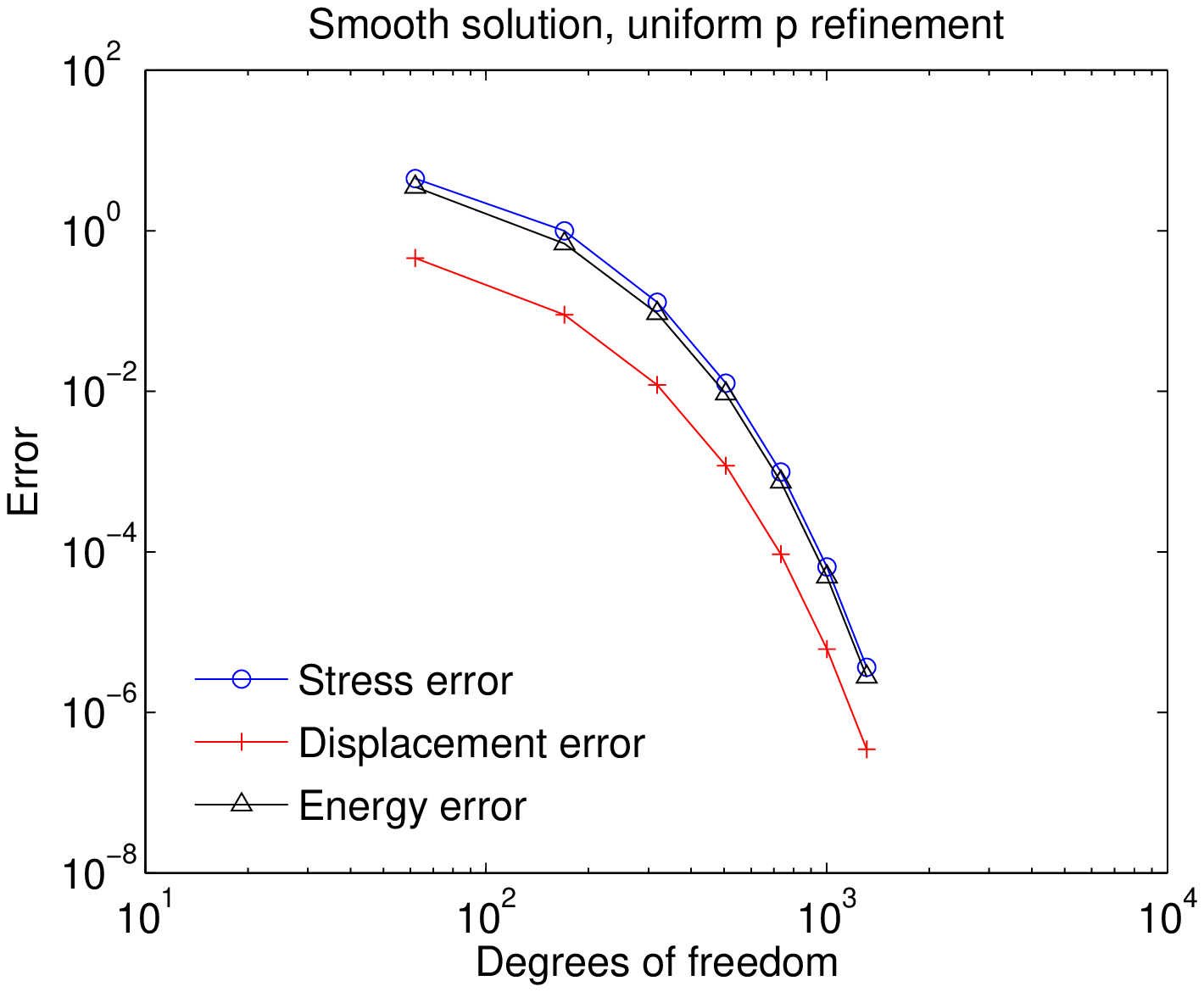}
    \label{fig:urefsmooth:p}
  }
\caption{Uniform refinement strategies for the smooth problem. 
  \label{fig:urefsmooth}}
\end{center}
\end{figure}
\begin{figure}
  \begin{center}
    \subfigure[Uniform h refinements.]{
      \includegraphics[width=0.47\textwidth]{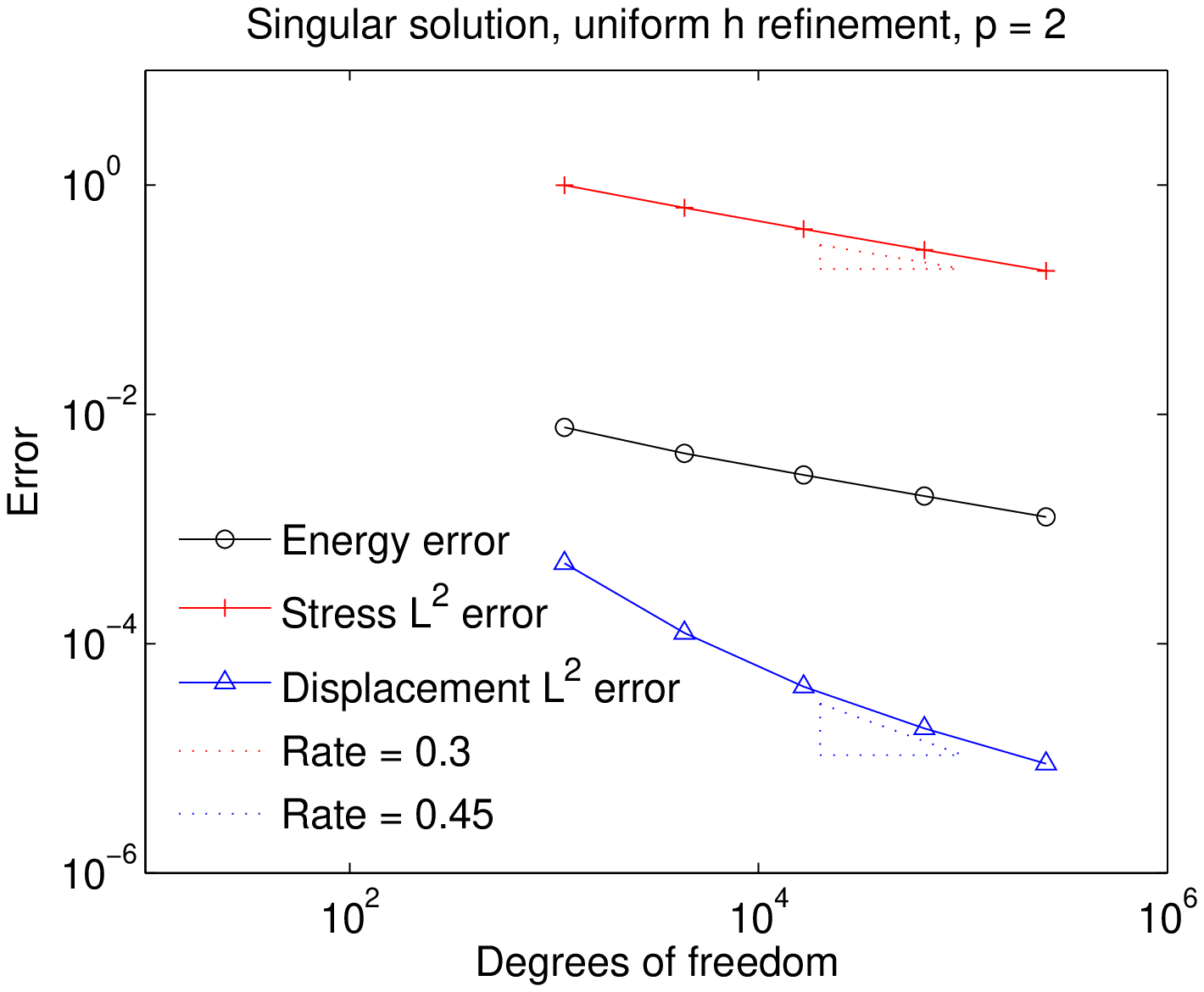}
      \label{fig:ureflshape:h}
    }
    \subfigure[Uniform p refinements.]{
      \includegraphics[width=0.47\textwidth]{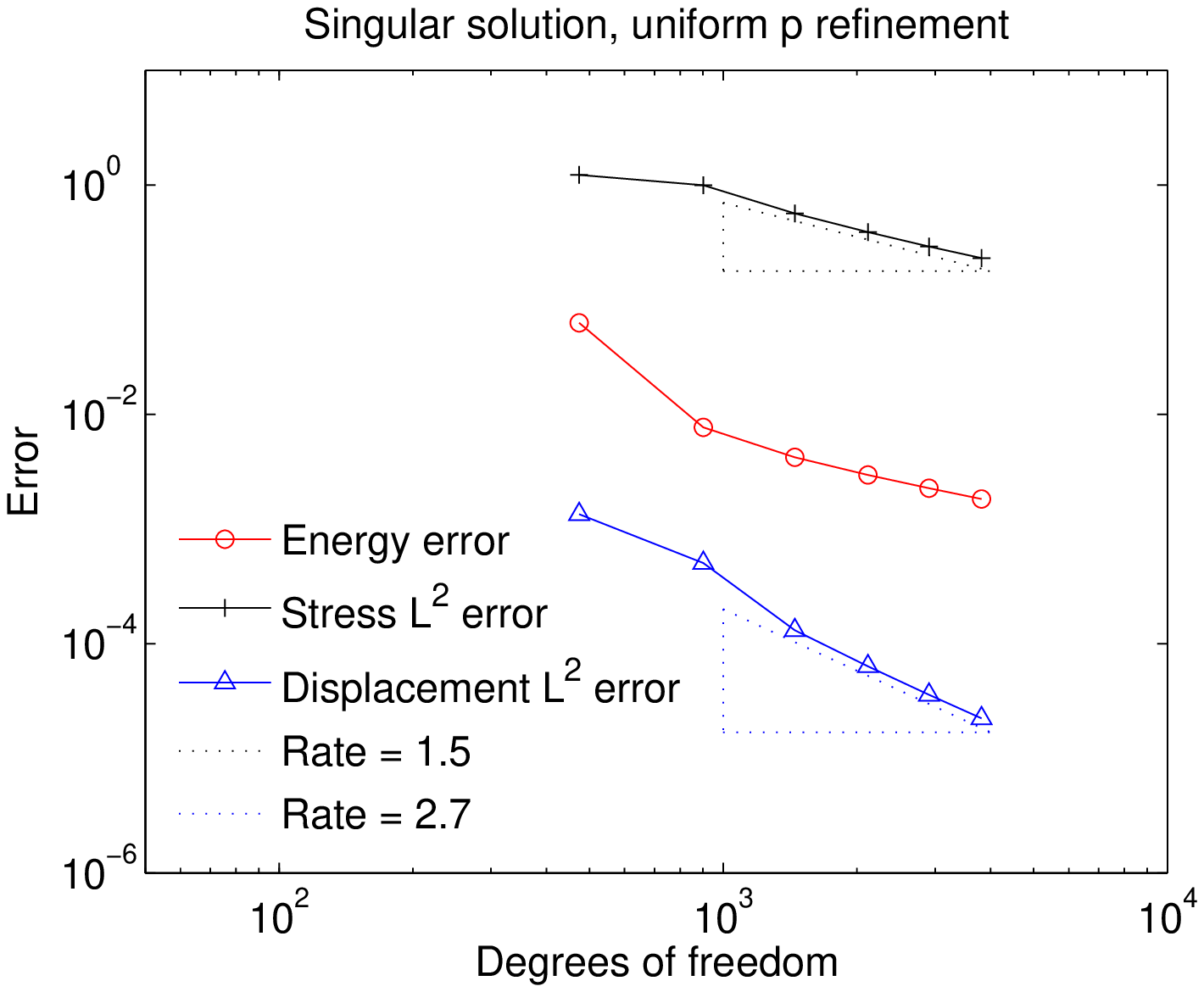}
      \label{fig:ureflshape:p}
    }
    \caption{Uniform refinement strategies for the L-shaped domain. 
      \label{fig:ureflshape}}
  \end{center}
\end{figure}

The observed decrease of the error as the degrees of freedom increase
is shown in Figure~\ref{fig:urefsmooth} for the smooth solution case
and in Figure~\ref{fig:ureflshape} for the L-shaped domain.  Note that
when we report the ``$L^2$ error'', we only consider the $L^2$-norm of
the errors in $u$ and $\sigma$, not the error in numerical fluxes or
traces.  

Consider the case of the smooth solution first. If we perform uniform
$h$-refinements, the number of degrees of freedom ($N$) is
$O(h^{-2})$. From Corollary~\ref{cor:rates}, we expect to see the
error decrease by $O(h^{p+1})$ for the smooth solution case, i.e.,
$O(N^{-(p+1)/2})$ in terms of $N$. This is confirmed in
Figure~\ref{fig:urefsmooth:h}. Also, since both displacement and stress
are infinitely smooth, they converge at the same rate. For uniform
$p$-refinements, exponential convergence is observed in
Figure~\ref{fig:urefsmooth:p}.

In the singular case of the L-shaped domain,
Figure~\ref{fig:ureflshape} shows the observed convergence history for
uniform refinements.  Since the stress variables are in
$H^{0.6038-\epsilon}$, we expect that the best approximation error for
stress should decrease at rate $h^{0.6038}$, or $N^{-0.3019}$. This is
in agreement with Figure~\ref{fig:ureflshape:h}. Additionally, since
displacement is in $H^{1.6038-\epsilon}$, one might think that its
best approximation error should decrease more or less at rate
$h^{1.6038}$, or $N^{-0.8019}$. However in the DPG method the errors
for both these variables are coupled together. So, while we observe
the optimal convergence rate for the stress variable, the convergence
rate for the displacement seems to be somehow limited by the
convergence rate of the stress. For uniform $p$-refinements, because
we are considering a singular solution, the convergence rate is
limited by the regularity of the solution, so unlike
Figure~\ref{fig:urefsmooth:p}, no exponential convergence is observed
in Figure~\ref{fig:ureflshape:p}.

\subsection{Comparison with the weakly symmetric mixed method.}

The smooth solution problem was also implemented using the weakly
symmetric mixed element given in~\cite{ArnolWinth02}. This mixed
method uses polynomials of degree one higher than our DPG method for
the stress trial space. Expectedly therefore, the stress
approximations given by the mixed method were generally observed to be
superior in the $L^2$ norm.  For the displacement however, both
methods use the same space, so it is interesting to compare the
displacement errors. This is done in Figure~\ref{fig:compare}. The DPG
method delivers lower displacement errors in the higher order case. In
the lowest order case (not shown in the figure) the mixed method
performs slightly better.

\begin{figure}
\begin{center}
\subfigure[]{
\includegraphics[width=0.47\textwidth,angle=0]{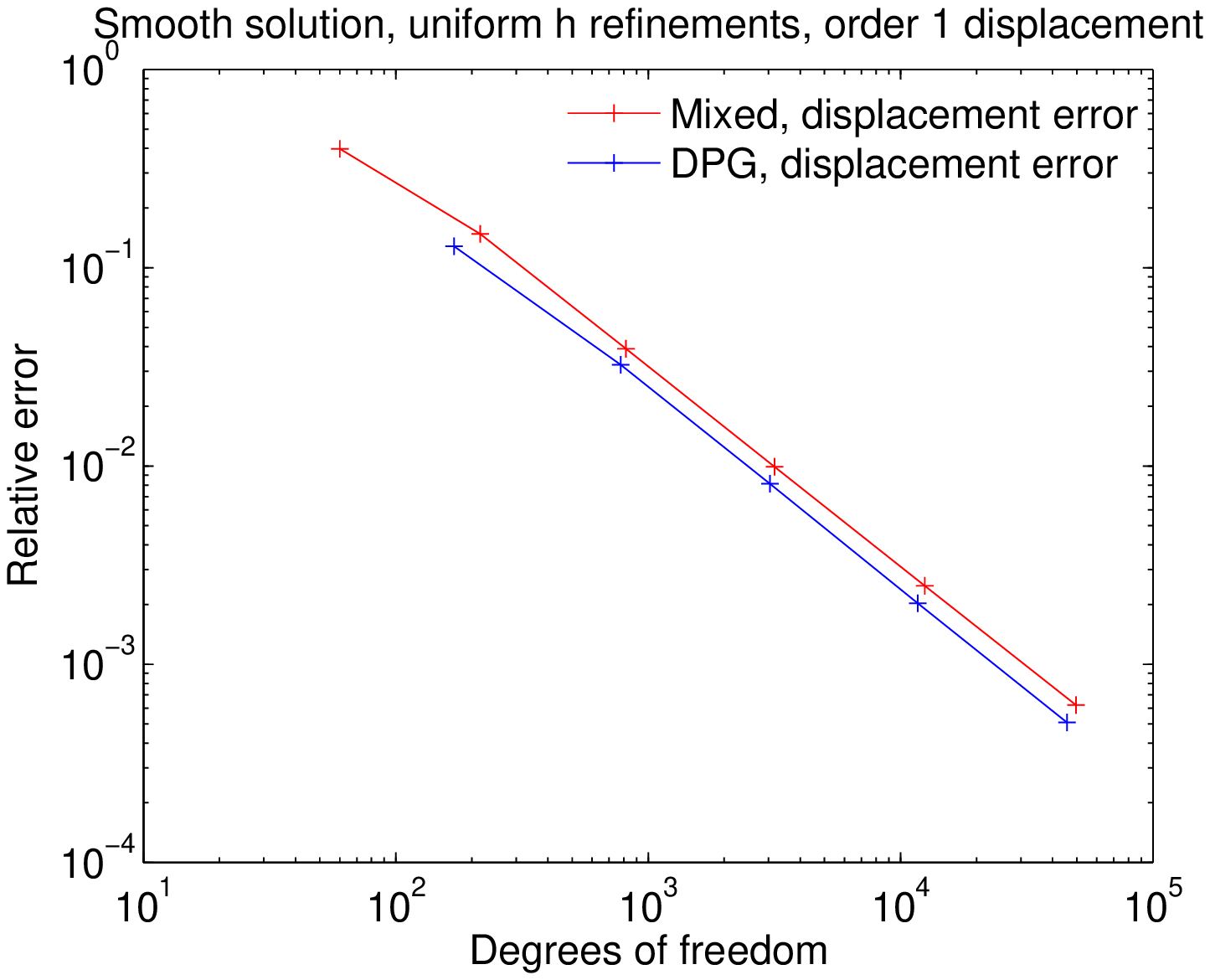}
\label{fig:compare:DPG}
}
\subfigure[]{
\includegraphics[width=0.47\textwidth,angle=0]{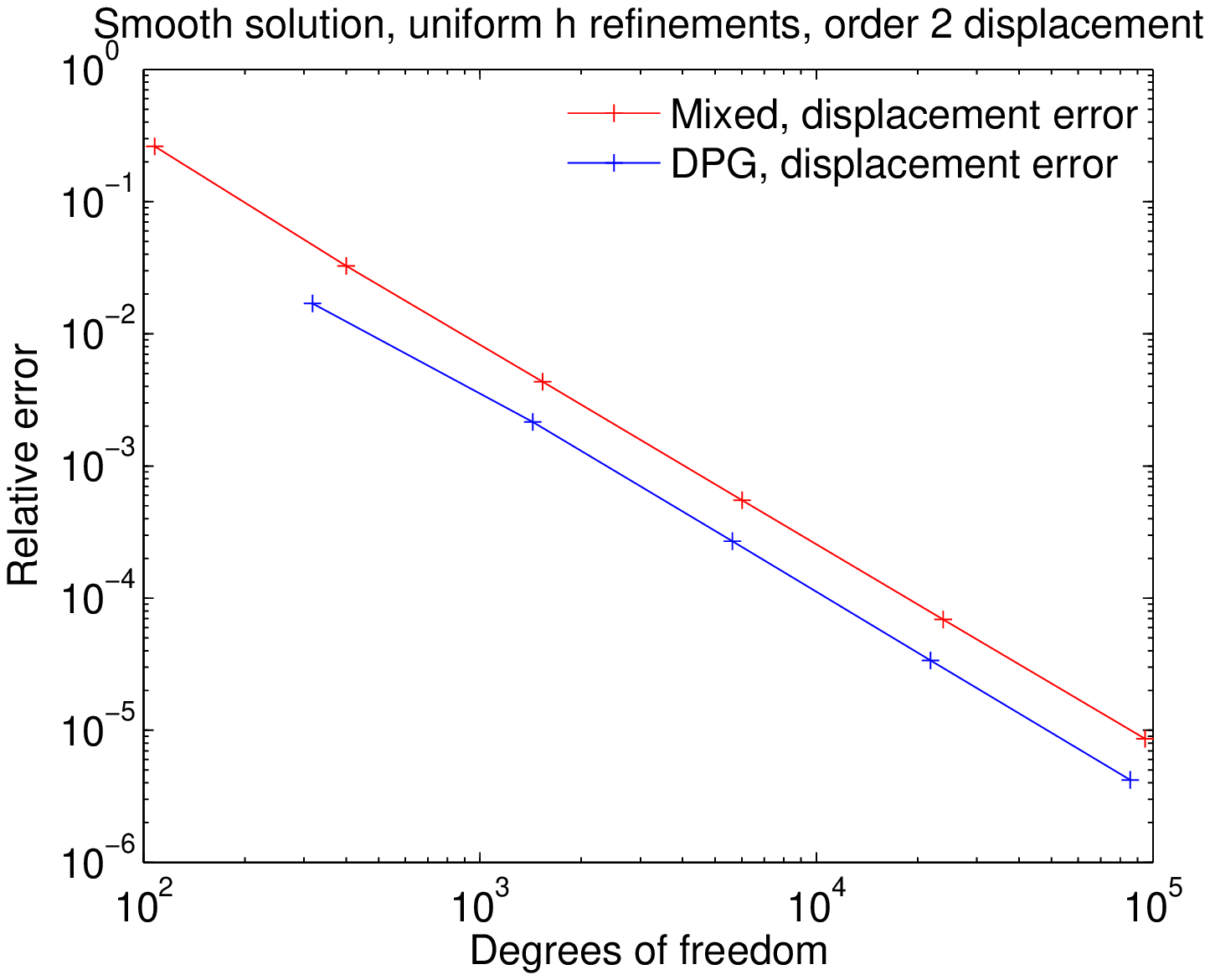}
}
\caption{The DPG method vs.~the mixed method. \label{fig:compare}}
\end{center}
\end{figure}

\subsection{Locking experiments}

In Figure~\ref{fig:lock} we show numerical evidence of the
locking-free property of the DPG method. The figure shows convergence
curves for various values of Poisson ratio close to the limiting
value of $0.5$.  We used piecewise bilinear elements with homogeneous
material data.  The convergence curves in Figure~\ref{fig:lock1} show
hardly any difference as $\nu$ approaches $0.5$. To be clearer, we
also plot the ratio of the $L^2$ discretization error to the best
approximation error (in $\sigma$ and $u$ combined) in
Figure~\ref{fig:lock2}. We see that the ratio remains close to the
optimal value of $1.0$ even as $\nu$ approaches $0.5$.


\begin{figure}
\begin{center}
\subfigure[Convergence curves for various $\nu$]{
  \includegraphics[width=0.47\textwidth,angle=0]{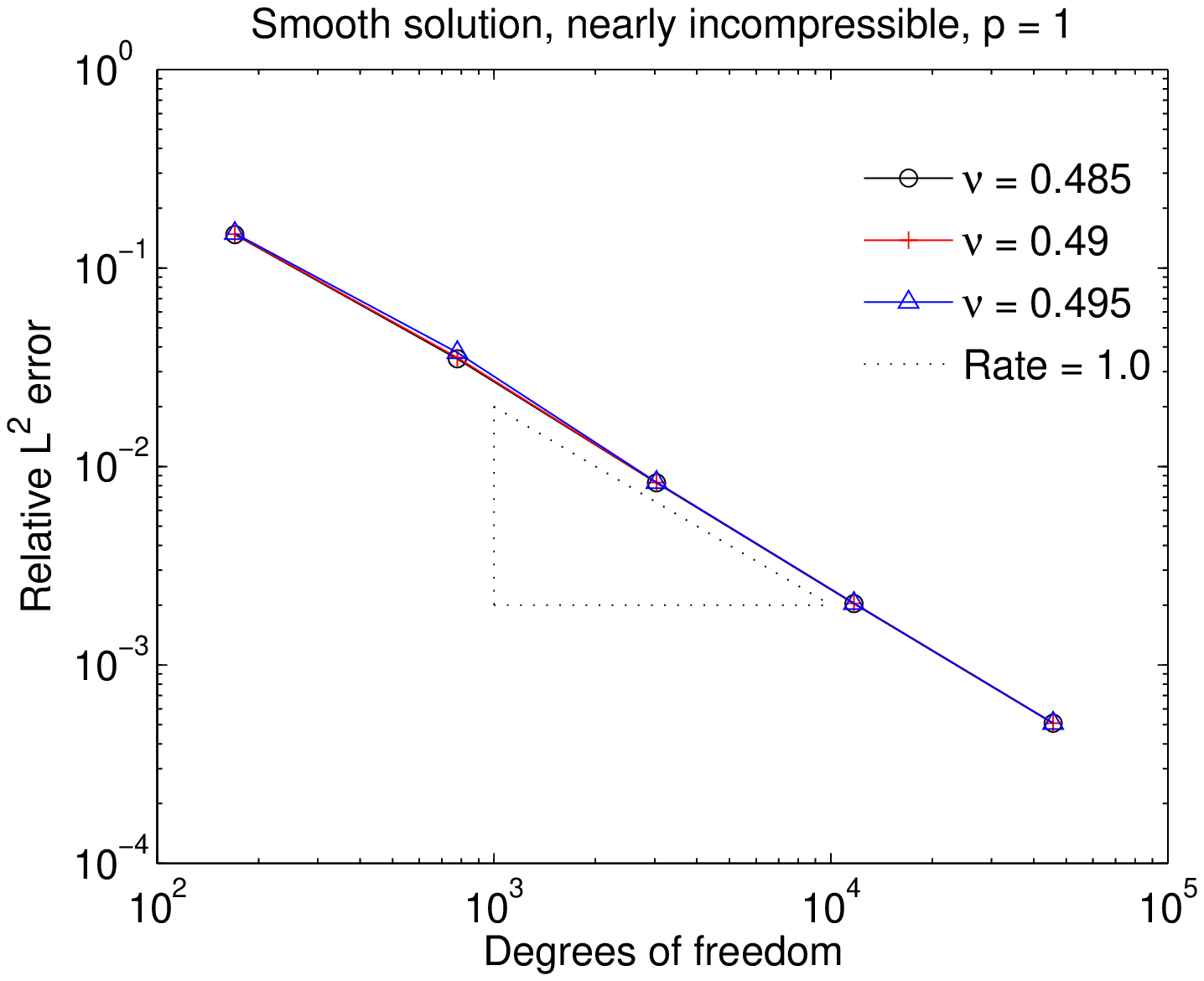}
  \label{fig:lock1}
} \subfigure[The ratio of the discretization error to the error in
best approximation as a function of $\nu$]{
  \includegraphics[width=0.47\textwidth,angle=0]{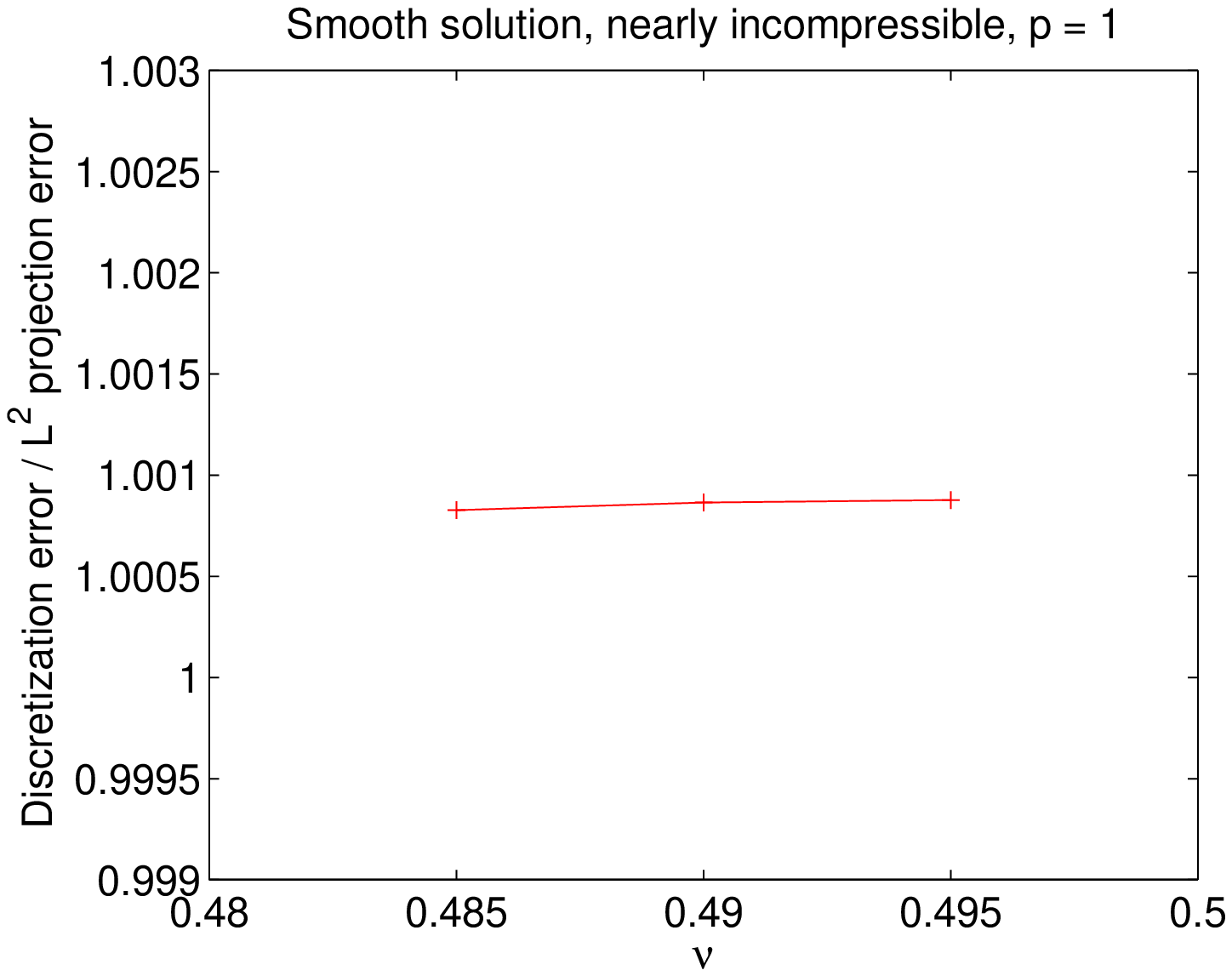}
  \label{fig:lock2}
}
\caption{Illustration of locking-free convergence (smooth data case).
  \label{fig:lock}}
\end{center}
\end{figure}

\subsection{Adaptivity}

All our adaptive schemes are based on the ``greedy'' strategy described
in~\cite{DemkoGopal:DPGanl}. This means that all elements which
contribute $50 \% $ of the maximum element error to the previously
described error representation function are marked for refinement. For
$hp$-adaptivity, we used the strategy suggested
in~\cite{AinswSenio97}, i.e., if an element contains the singularity,
it is $h$-refined, otherwise it is $p$-refined.

Figure~\ref{fig:adapt} shows results from both adaptivity
schemes. Note that a nearly optimal rate of $O(N^{-1.2})$is observed
for the $h$-adaptivity scheme. The $hp$-adaptive scheme results in an
optimal rate of $O(N^{-1.5})$. Finally, Figure~\ref{fig:mesh} shows
the $hp$ mesh obtained after 12 iterations and Figure~\ref{fig:sol:ux}
shows one component of the corresponding solution. The group relative
$L^2$ error is reduced to $0.9 \% $.

\begin{figure}
\begin{center}
  \subfigure[Adaptive $h$ refinements.]{
    \includegraphics[width=0.47\textwidth,angle=0]{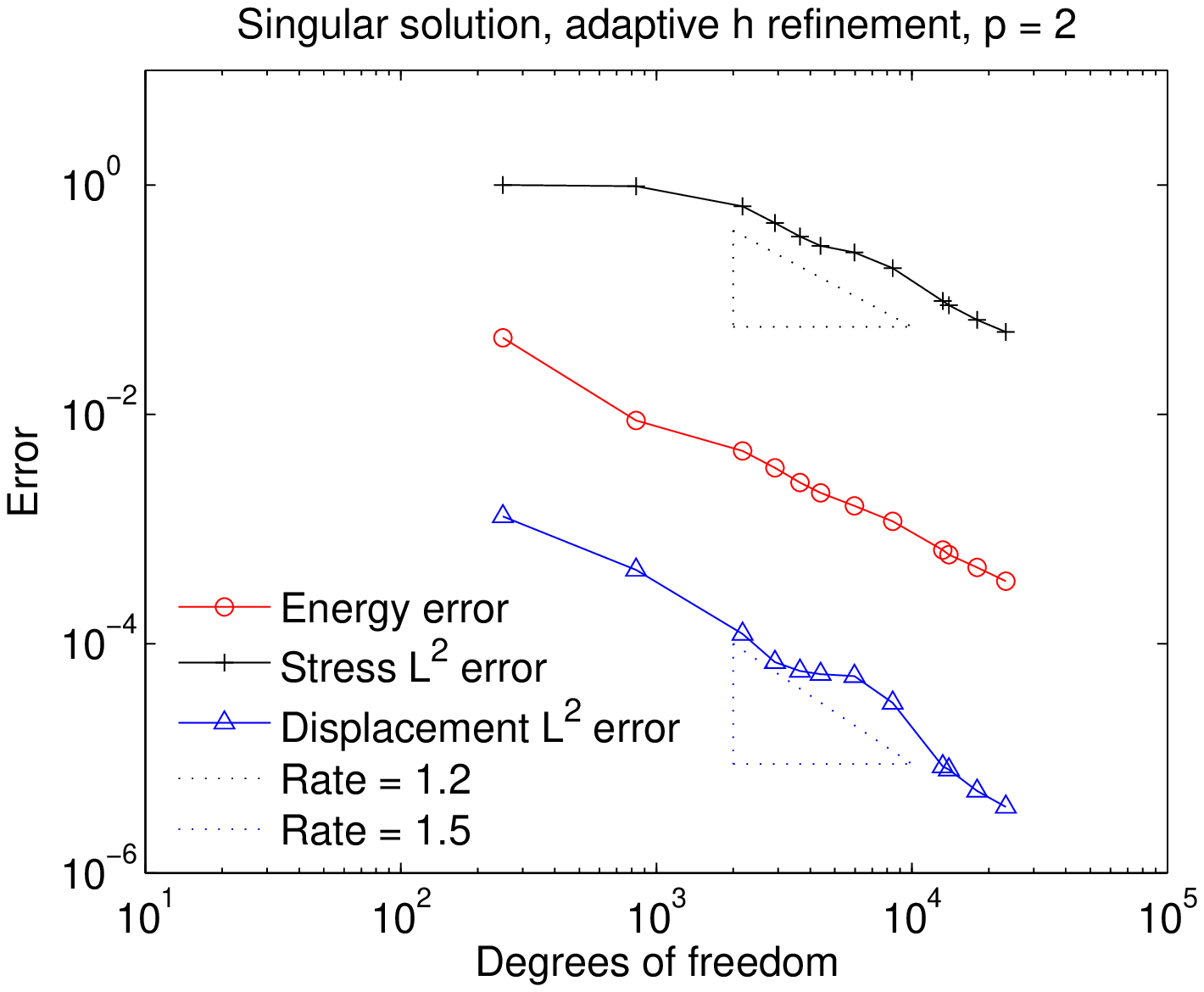}
  }
  \subfigure[Adaptive $hp$ refinements.]{
    \includegraphics[width=0.47\textwidth,angle=0]{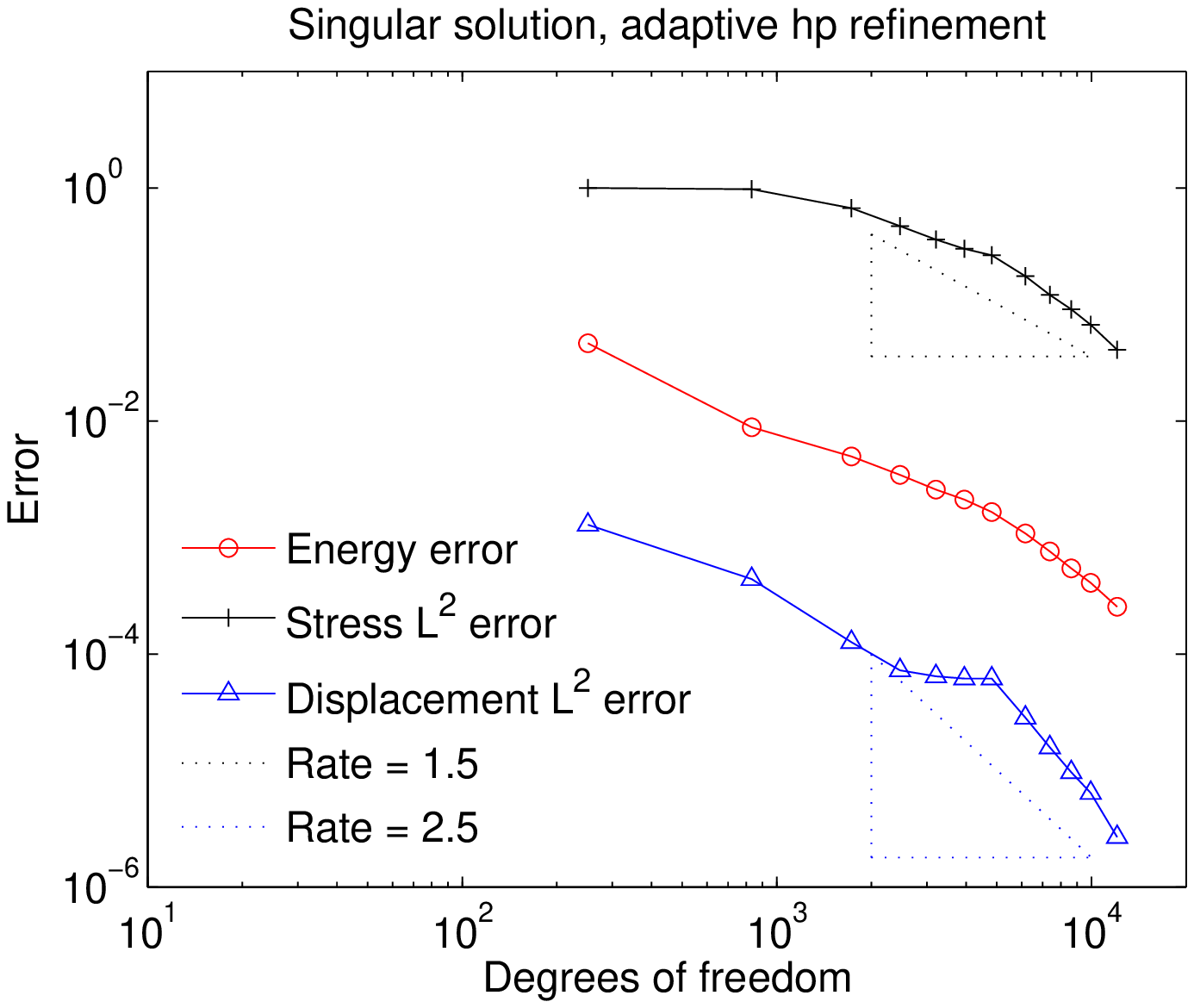}
  }

  \bigskip

  \subfigure[A comparison of refinement strategies.]{
    \includegraphics[width=0.47\textwidth,angle=0]{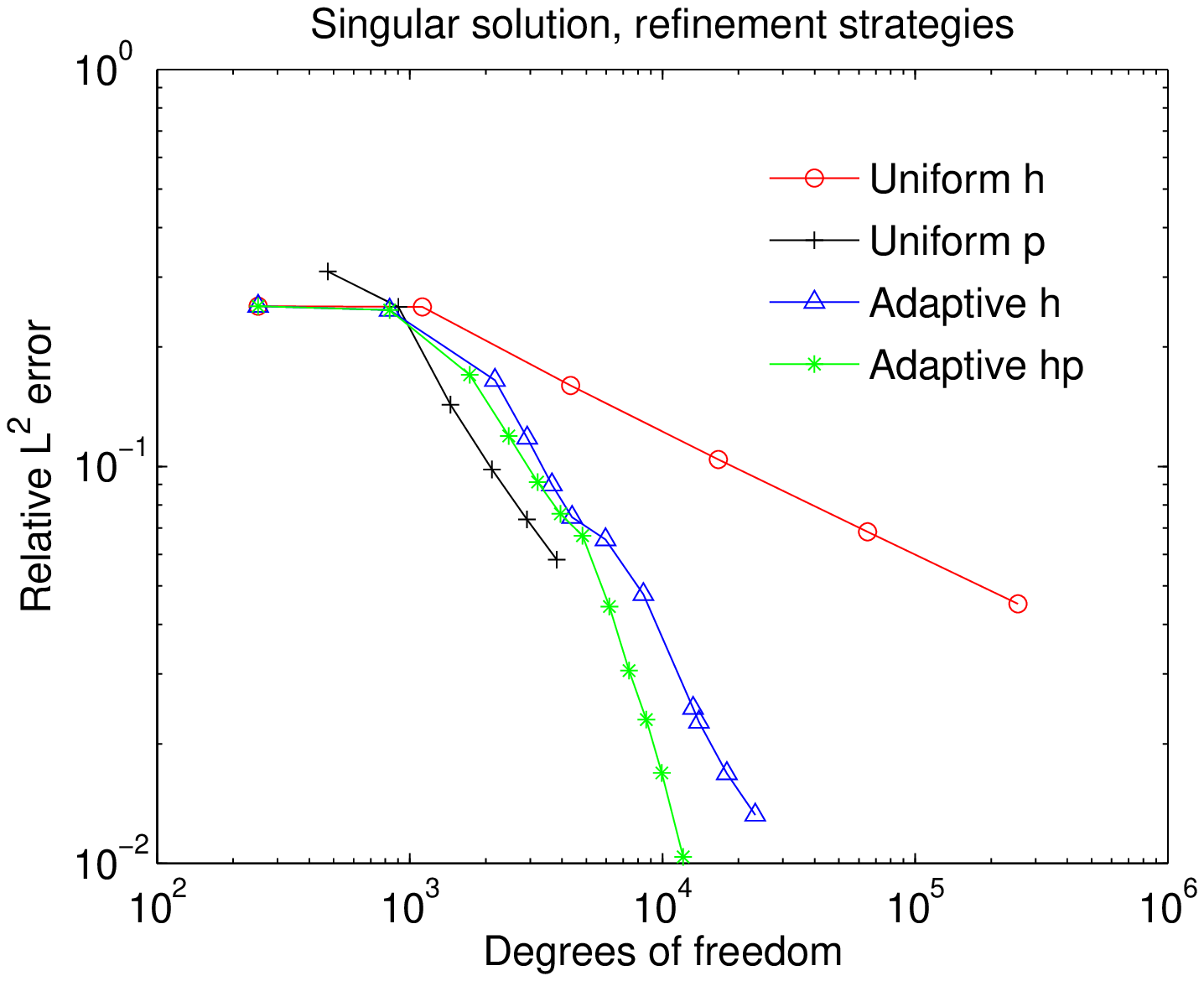}
  }
  \subfigure[Various optimal test function approximations.]{
    \includegraphics[width=0.47\textwidth,angle=0]{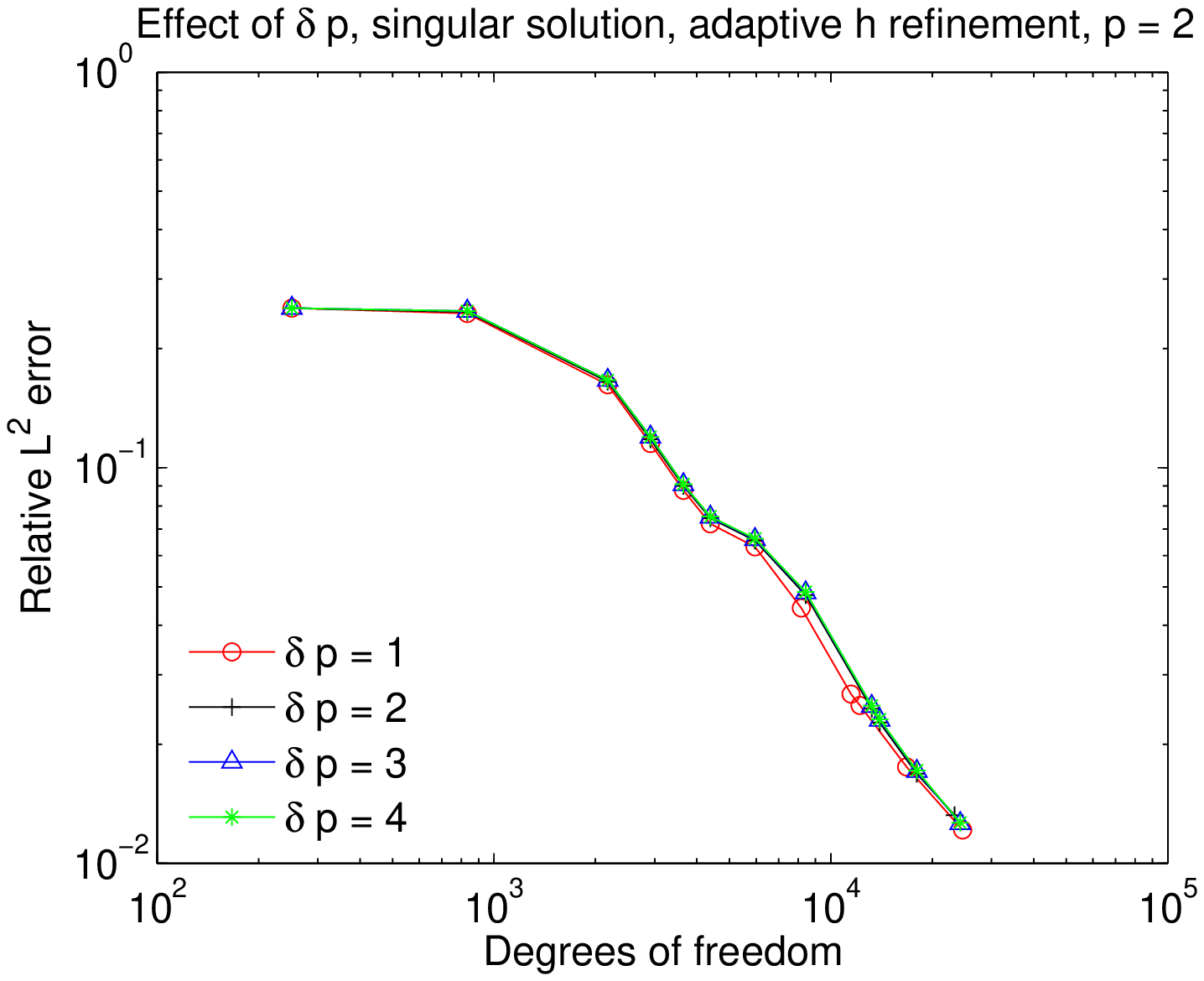}
    \label{fig:approx}
  }

  \bigskip

  \subfigure[The $hp$ mesh after 12 iterations. 
  Element degrees are represented by color. (The color scale is tied to $p+1$.)]{
    \hspace{-0.5cm}\includegraphics[width=0.32\textwidth,angle=90]{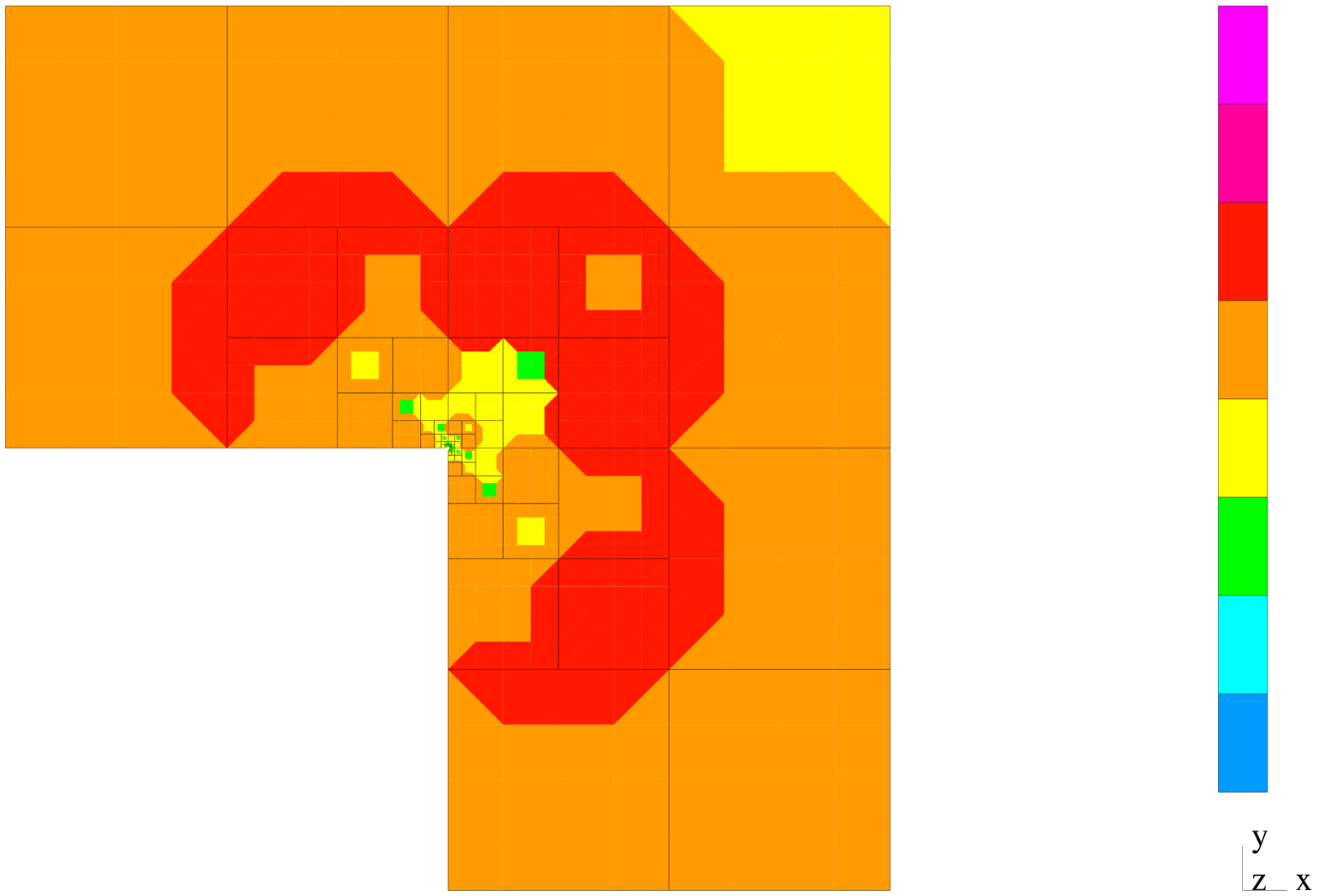}
    \label{fig:mesh}
  }\qquad
  \subfigure[The $x$-component of the computed displacement~($u_x$). The color scale  indicates value of $u_x$.]{
    \includegraphics[width=0.32\textwidth,angle=90]{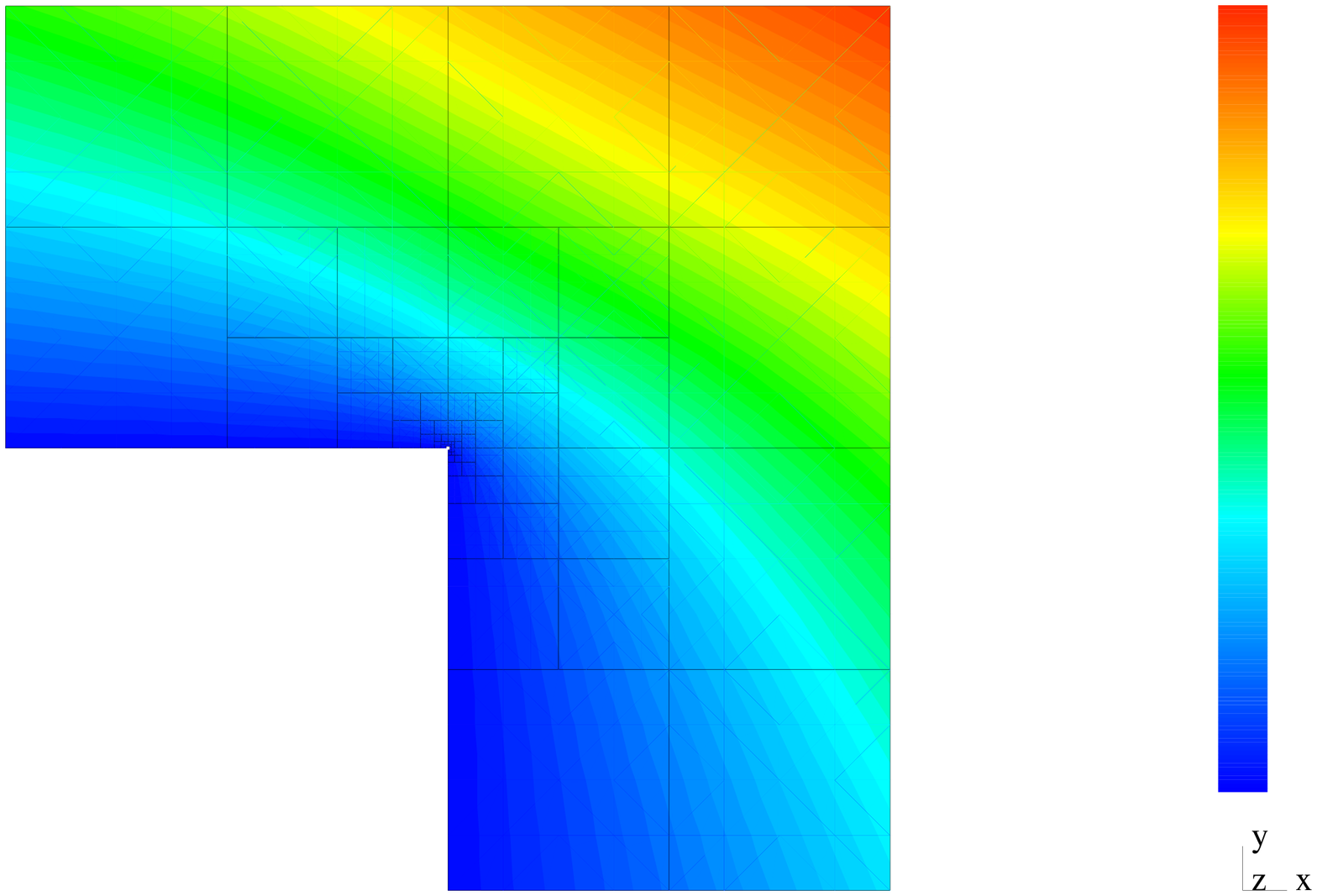}
    \label{fig:sol:ux}
  }
 \caption{Results from the adaptive scheme for the L-shaped domain. \label{fig:adapt}}
\end{center}
\end{figure}

\subsection{Approximation of optimal test functions}

Figure~\ref{fig:approx} shows the effect of $\delta p$ as seen in the
$h$-adaptive process for the L-shaped domain problem.  This measures
the effect of approximating optimal test functions using the operator
$\tilde{T}$ as opposed to $T$, Since the curves for $\delta p =2,3,4$
are coincident, it appears that we are sufficiently approximating the
optimal test functions.



\appendix

\section{A property of the weakly symmetric mixed formulation}
\label{sec:weakly-symm}

We consider a mixed method for linear elasticity with weakly imposed
stress symmetry. The method we consider differs from a standard
method~\cite{ArnolFalkWinth07} only in that it has an extra Lagrange
multiplier. It is well known that the mixed formulation does not
lock~(see e.g.,\cite{BrezziFortin1991,CarstDolzmFunke00,Stenb88}) for
homogeneous isotropic material parameters. In this appendix, we will
provide a stability result for slightly more general materials. Note
however, that the main goal of this appendix is to establish stability
estimates for the mixed method in the form needed for the analysis of
the DPG scheme in the earlier sections.

The formulation reads as follows: Find $(\sigma,u,\rho,a)\in
H(\text{div},\Omega;\mathbb{M}) \times L^{2}(\Omega;\mathbb{V}) \times
L^{2}(\Omega;\mathbb{K}) \times \mathbb{R}$ satisfying
\begin{subequations}
\label{eq:bal}
\begin{align}
\label{eq1_balanced_dmf}
 (A\sigma,\tau)_{\Omega}+(u,\dive\tau)_{\Omega}
 +(\rho,\tau)_\Omega 
 + (aQ_{0}^{-1}AI,\tau)_{\Omega}
& =(F_{1},\tau)_{\Omega},
\\
\label{eq2_balanced_dmf}
(\dive\sigma,{v})_{\Omega}
& =(F_{2},{v})_{\Omega},
\\
\label{eq3_balanced_dmf}
(\sigma,\eta)_{\Omega}
& =(F_{3},q)_{\Omega},
\\
\label{eq4_balanced_dmf}
(\sigma,bQ_{0}^{-1}AI)_{\Omega}
& =(F_{4},bI)_{\Omega},
\end{align}
\end{subequations}
for all $(\tau, {v}, \eta,b) \in H(\text{div},\Omega;\mathbb{M}),
\times L^{2}(\Omega;\mathbb{V}) \times L^{2}(\Omega;\mathbb{K}) \times
\mathbb{R}.$ Recall that $Q_0$ is as defined in~\eqref{Q0_def}. This
formulation, specifically~\eqref{eq4_balanced_dmf}, is motivated by
the same constraint that motivated the second DPG method,
namely~\eqref{zero_trace_eq}.

\begin{theorem} \label{thm:balanced_dmf} %
  Let $(F_{1},F_{2},F_{3},F_{4})\in L^{2}(\Omega;\mathbb{M}) \times
  L^{2}(\Omega;\mathbb{V})\times L^{2}(\Omega;\mathbb{M})\times
  L^{2}(\Omega;\mathbb{M})$. Then:
  \begin{enumerate}
  \item Problem~\eqref{eq:bal} is uniquely solvable and the solution
    component $u$ is in fact in~$H^{1}_{0}(\Omega;\mathbb{V})$.

  \item There is a positive constant $C_{0}$ such that
    \begin{equation}
      \label{bound_balanced_dmf}
      \Vert\sigma\Vert_{H(\text{div},\Omega)}+\Vert u\Vert_{H^{1}(\Omega)}+\Vert\rho\Vert_{\Omega}+\vert a\vert
      \leq C_{0} \big(\Vert F_1\Vert_{\Omega}
      +\Vert F_2\Vert_{\Omega}
      +\Vert F_3\Vert_{\Omega}
      +\Vert F_4\Vert_{\Omega}\big).
    \end{equation}

  \item In addition, if Assumption~\ref{asm:iso} holds, then the
    constant $C_{0}$ in~\eqref{bound_balanced_dmf} takes the
    form
    \begin{equation}
      \label{uniform_bound_dmf}
      C_{0} = \bar{c}_{1}P_{0}^{-1}B^{4}(\Vert A\Vert+P_{0}+1)^{2}(\Vert A\Vert+B),
    \end{equation}
    where $\bar{c}_{1}$ is a positive constant independent of $A$.

  \end{enumerate}
\end{theorem}

We will use the \Babuska-Brezzi theory~\cite{BrezziFortin1991} and
results from~\cite{ArnolFalkWinth07} to prove this theorem. In order
to verify the conditions of the theory, we will use the following
lemma.

\begin{lemma}
\label{lemma_uniform_trace}
If Assumption~\ref{asm:iso} holds, then there is a positive constant
$\bar{c}_{0}$ independent of the material coefficient~$A$ such that
\begin{equation}
\label{uniform_trace}
\Vert\trace\varphi\Vert_{\Omega}^2
\leq
\,\bar{c}_{0}^2 B^2
\left(
  \Vert\varphi_{D}\Vert_{\Omega}^{2}+\Vert\dive\varphi\Vert_{\Omega}^{2}
\right),
\end{equation} 
for any $\varphi\in H(\text{div},\Omega;\mathbb{M})$ which satisfies
\begin{equation}
\label{A_trace}
\int_{\Omega}\trace (A\varphi)=0.
\end{equation}
\end{lemma}

\begin{proof}
  This proof is similar to a proof in~\cite{BrezziFortin1991}.  We can
  apply a standard regular right inverse of divergence to
  $\tr(A\varphi)$ since~\eqref{A_trace}. Hence there exists a constant
  $c_{0}>0$ and $\eta\in H_{0}^{1}(\Omega;\mathbb{V})$ such that
  \begin{equation*}
    \dive\eta = Q_{0}^{-1}\trace (A\varphi),
    \quad
    \Vert\eta\Vert_{H^{1}(\Omega)}\leq c_{0}Q_{0}^{-1}
    \Vert \trace (A\varphi)\Vert_{\Omega}.
  \end{equation*}
  By the isotropy assumption -- see~(\ref{PQ_decomp}) -- we have that
  $\trace (A\varphi) = Q\trace\varphi$. This implies that
  \begin{equation*}
    \dive\eta
    = (QQ_{0}^{-1})\trace\varphi,\quad \Vert\eta\Vert_{H^{1}(\Omega)}\leq c_{0}Q_{0}^{-1}
    \Vert Q\trace\varphi\Vert_{\Omega}.
  \end{equation*}
  Then, since $Q Q_0^{-1}\ge 1$ a.e.,~we have 
  \begin{align*}
    \Vert\trace\varphi\Vert_{\Omega}^{2}
    & \leq 
    ((QQ_{0}^{-1})\trace\varphi,\trace\varphi)_{\Omega}
    =(\dive\eta,\trace\varphi)_{\Omega}
    =( (\dive\eta )I,\varphi)_{\Omega}
    \\
    &=N(\nabla\eta-(\nabla\eta)_{D},\varphi)_{\Omega}
    = -N(\eta,\dive\varphi)_{\Omega}-N((\nabla\eta)_{D},\varphi)_{\Omega}
    \\ 
    & = -N(\eta,\dive\varphi)_{\Omega}-N(\nabla\eta,\varphi_{D})_{\Omega}
    \\
    & \leq N\Vert\eta\Vert_{H^{1}(\Omega)}
    \left( \Vert\varphi_{D}\Vert_{\Omega}^{2}+\Vert\dive\varphi\Vert_{\Omega}^{2}
      \right)^{1/2}
    \\
    & \leq c_{0}NQ_{0}^{-1}\Vert Q\trace\varphi\Vert_{\Omega}
    \left( 
      \Vert\varphi_{D}\Vert_{\Omega}^{2}+\Vert\dive\varphi\Vert_{\Omega}^{2}.
    \right)^{1/2}
\end{align*}
Setting  $\bar{c}_{0}=c_{0}N$, the lemma is proved. 
\end{proof}

\begin{proof}[Proof of Theorem~\ref{thm:balanced_dmf}]
  To apply the \Babuska-Brezzi theory, we need to verify two
  conditions: (i)~the coercivity on kernel, and (ii)~the inf-sup
  condition.

  Step~(i).~{\em Coercivity on kernel:} Define the kernel space
  \[
  V_{0}=\{\tau\in
  H(\text{div},\Omega;\mathbb{M}):\dive\tau=0, \; \tau'=\tau, \;
  \int_{\Omega} \trace (A\tau)=0\}.
  \]
  Clearly, if $A$ is uniformly coercive, then there is a positive
  constant $c_{1}$, depending on $A$, such that
  \begin{equation}
    \label{S1}
    c_{1}\Vert \tau\Vert_{H(\text{div},\Omega)}^{2}\leq (A\tau,\tau)_{\Omega},
    \qquad\forall \tau \in V_0.
  \end{equation} 
  
  If in addition Assumption~\ref{asm:iso} holds, then we can give the
  dependence of $c_1$ on $P_0$ and~$Q$, as we see now. For any
  $\tau\in V_{0}$,
  \begin{equation*}
    (A\tau,\tau)_{\Omega}\geq (P\tau_{D},\tau_{D})_{\Omega}
    \geq P_{0}\Vert\tau_{D}\Vert_{\Omega}^{2}.
  \end{equation*}
  Using Lemma~\ref{lemma_uniform_trace} and the fact that
  $\dive\tau=0$, we have that
\begin{equation*}
\Vert\tau_{D}\Vert_{\Omega}^{2}\geq (\bar{c}_{0}B)^{-2}\Vert\trace\tau\Vert_{\Omega}^{2}.
\end{equation*} 
Since $\Vert\tau\Vert_{\Omega}^{2}=\Vert\tau_{D}\Vert_{\Omega}^{2}+N^{-1}\Vert\trace\tau\Vert_{\Omega}^{2}$, 
we have that
\begin{equation*}
  \Vert\tau\Vert_{\Omega}^{2}\leq (1+N^{-1}\bar{c}_{0}^{2}B^{2})\Vert\tau_{D}\Vert_{\Omega}^{2}.
\end{equation*}
So, for any $\tau\in V_{0}$, we have that
\begin{equation*}
\dfrac{P_{0}}{1+N^{-1}\bar{c}_{0}^{2}B^{2}}\Vert\tau\Vert_{H(\text{div},\Omega)}^{2}
=
\dfrac{P_{0}}{1+N^{-1}\bar{c}_{0}^{2}B^{2}}\Vert\tau\Vert_{\Omega}^{2}
\le 
(A\tau,\tau)_{\Omega}
\end{equation*}
and we conclude that (\ref{S1}) holds with
\begin{equation}
  \label{eq:3}
  c_1 =\dfrac{P_{0}}{1+\bar{c}_{0}^{2}B^{2}}
\end{equation}
in the isotropic case.

Step~(ii).~{\em Inf-sup condition:} The inf-sup condition will follow once
we show that there is a positive constant $c_{2}$ such that for any
$(u,\rho,a)\in L^{2}(\Omega;\mathbb{V})\times
L^{2}(\Omega;\mathbb{K})\times \mathbb{R}$, there is a $\tau\in
H(\text{div},\Omega;\mathbb{M})$ satisfying
\begin{equation}
\label{S2}
(u,\dive\tau)_{\Omega}
+(\rho,\tau)_\Omega 
+ (aQ_{0}^{-1}AI,\tau)_{\Omega}
\geq c_{2}
\Vert\tau\Vert_{H(\text{div},\Omega)}
\left(\Vert u\Vert_{\Omega}+\Vert \rho\Vert_{\Omega}+\vert a\vert\right).
\end{equation}

To this end, we first recall~\cite[Theorem
$11.1$]{AFW:2006:ECH}. Accordingly, there is a $\tau_{0}\in
H(\text{div},\Omega;\mathbb{M})$ and $c_3>0$ such that
$\dive\tau_{0}=u$, $\skw\tau_{0}=\rho$, and
\begin{equation}
\label{S2_bound1}
\Vert\tau_{0}\Vert_{H(\text{div},\Omega)}\leq 
c_{3}(\Vert u\Vert_{\Omega}+\Vert\rho\Vert_{\Omega}).
\end{equation}
The constant $c_{3}$ depends only on $\Omega$. 

To prove~\eqref{S2},  we choose $\tau$ of the form 
$\tau=\tau_{0}+\lambda I$ where $\lambda\in\mathbb{R}$. 
Obviously, 
\[
\tau\in H(\text{div},\Omega,\mathbb{M}), \quad \dive\tau=u,\quad
\text{and}\quad\skw\tau=\rho,
\]
for any $\lambda \in \RRR$.  So, to show the estimate~\eqref{S2}, we
need only choose $\lambda\in\mathbb{R}$ such that
\begin{equation*}
(aQ_{0}^{-1}AI,\tau)_{\Omega} =(aQ_{0}^{-1}AI,\tau_{0}+\lambda I)_{\Omega}
=\vert a\vert^{2},
\end{equation*}
i.e.,
\begin{equation}
\label{S2_bound2}
\lambda=\dfrac{a-Q_{0}^{-1}(AI,\tau_{0})_{\Omega}}{Q_{0}^{-1}\int_{\Omega}\trace (AI)}.
\end{equation}
Then, by~(\ref{S2_bound1}), there is a positive constant $c_{2}$ such
that (\ref{S2}) holds.

If the material is isotropic, then the dependence of $c_2$ on the
components of $A$ can be tracked, as follows. Observe that since $Q
Q_0^{-1}\ge 1$, we have
\begin{gather*}
 Q_{0}^{-1}\int_{\Omega}\trace (AI)
=N\int_{\Omega}Q_{0}^{-1}Q 
\geq N\vert\Omega\vert,
\\
\vert Q_{0}^{-1}(AI,\tau_{0})_{\Omega}\vert
=\left| \int_{\Omega} (Q_{0}^{-1}Q)\trace\tau_{0}\right|
\leq c_{4}B\Vert\tau_{0}\Vert_{\Omega},
\end{gather*}
with a constant $c_4$ depending only on $\Omega$. Using this
in~\eqref{S2_bound2}, we have 
\begin{align*}
\Vert\tau\Vert_{H(\text{div},\Omega)}
& \leq 
\Vert\tau_0\Vert_{H(\text{div},\Omega)} + \Vert \lambda I\Vert_{H(\text{div},\Omega)} 
\\
& \le 
c_{3}(\Vert u\Vert_{\Omega}+\Vert\rho\Vert_{\Omega})
+
\dfrac{c_{3}c_{4}B
(\Vert u\Vert_{\Omega}+\Vert\rho\Vert_{\Omega})+\vert a\vert}
{\sqrt{N\vert\Omega\vert}}.
\end{align*}
Thus there is a constant $c_5$ depending only on $\Omega$ such that
\begin{equation*}
\Vert\tau\Vert_{H(\text{div},\Omega)}\leq c_{5}B(\Vert u\Vert_{\Omega}
+\Vert\rho\Vert_{\Omega}+\vert a\vert),
\end{equation*}
and therefore (\ref{S2}) holds with 
\begin{equation}
  \label{eq:4}
c_{2}=(c_{5}B)^{-1}
\end{equation}
in the isotropic case.  

Step~(iii). Since we have verified the two conditions of the
\Babuska-Brezzi theory, we conclude that there is a unique solution
$(\sigma,u,\rho,a)\in H(\text{div},\Omega;\mathbb{M})\times
L^{2}(\Omega;\mathbb{V}) \times L^{2}(\Omega;\mathbb{K}) \times
\mathbb{R}$.  Moreover, the theory guarantees -- see
\cite[Eq.~(1.29)--(1.30)]{BrezziFortin1991} -- that the stability estimate 
\begin{equation*}
\Vert\sigma\Vert_{H(\text{div},\Omega)}+\Vert u\Vert_{\Omega}+\Vert\rho\Vert+\vert a\vert
\leq \tilde{C} (\Vert F_1\Vert_{\Omega}+\Vert F_2\Vert_{\Omega}+\Vert F_3\Vert_{\Omega}+\Vert F_4\Vert_{\Omega})
\end{equation*} 
holds with 
\begin{equation*}
\tilde{C}=c_{1}^{-1}c_{2}^{-2}(\Vert A\Vert+c_{1}+c_{2})^{2}.
\end{equation*}
If in addition, Assumption~\ref{asm:iso} holds, then by~\eqref{eq:3}
and~\eqref{eq:4},
\begin{equation}
\label{eq:5}
\tilde{C}=P_{0}^{-1}(1+\bar{c}_{0}^{2}B^{2})c_{5}^{2}B^{2}
(\Vert A\Vert+(c_{5}B)^{-1}+P_{0})^{2}.
\end{equation}

Step~(iv). To prove that $u$ is in fact in $H^1(\Omega,\VVV)$, we
observe that by choosing $\tau\in \mathcal{D}(\Omega;\mathbb{M})$
arbitrarily, we can conclude that the equality 
\begin{equation}
\label{eq5_balanced_dmf}
A\sigma-\nabla u+\rho + aQ_{0}^{-1}AI=F_{1}
\end{equation}
holds in the sense of distributions. Hence $u\in
H^{1}(\Omega;\mathbb{V})$.  Consequently, we may
integrate~(\ref{eq1_balanced_dmf}) by parts for any $\tau \in
H(\text{div},\Omega;\mathbb{M})$ and use~\eqref{eq5_balanced_dmf} to
conclude that $u\in H^{1}_{0}(\Omega;\mathbb{V})$.  By
(\ref{eq5_balanced_dmf}),
\begin{equation*}
\Vert\sigma\Vert_{H(\text{div},\Omega)}+\Vert u\Vert_{H^{1}(\Omega)}+\Vert\rho\Vert+\vert a\vert
\leq \tilde{C}(2+\Vert A\Vert+Q_{0}^{-1}\Vert AI\Vert_{\Omega})(\Vert F_1\Vert_{\Omega}+\Vert F_2\Vert_{\Omega}
+\Vert F_3\Vert_{\Omega}+\Vert F_4\Vert_{\Omega}).
\end{equation*}
We have thus proved (\ref{bound_balanced_dmf}) with
$C_{0}=\tilde{C}(2+\Vert A\Vert+Q_{0}^{-1}\Vert AI\Vert_{\Omega})$.

If in addition, the material is isotropic, then by~\eqref{eq:5}, the
constant $C_{0}$ can be written as
\begin{equation*}
C_{0}=P_{0}^{-1}(1+\bar{c}_{0}^{2}B^{2})c_{5}^{2}B^{2}
(\Vert A\Vert+(c_{5}B)^{-1}+P_{0})^{2}(2+\Vert A\Vert+N\vert\Omega\vert B).
\end{equation*}
Since $B\geq 1$, we conclude that there is a positive constant
$\bar{c}_{1}$ such that $ C_{0}\leq \bar{c}_{1}P_{0}^{-1}B^{4}(\Vert
A\Vert+P_{0}+1)^{2}(\Vert A\Vert+B),$ as stated
in~\eqref{uniform_bound_dmf}.
\end{proof}

\begin{acknowledgements}
This work was supported in part by the NSF under grant DMS-1211635. The authors also 
gratefully acknowledge the support from the IMA (Minneapolis) during their 2010-11 program.
\end{acknowledgements}


\begin{thebibliography}{}
%
%
\bibitem{AinswSenio97}
{\sc M.~Ainsworth and B.~Senior}, {\em Aspects of an adaptive {$hp$}-finite
  element method: adaptive strategy, conforming approximation and efficient
  solvers}, Comput. Methods Appl. Mech. Engrg., 150 (1997), pp.~65--87.
\newblock Symposium on Advances in Computational Mechanics, Vol. 2 (Austin, TX,
  1997).

\bibitem{ArnolAwanoWinth08}
{\sc D.~N. Arnold, G.~Awanou, and R.~Winther}, {\em Finite elements for
  symmetric tensors in three dimensions}, Math. Comp., 77 (2008),
  pp.~1229--1251.

\bibitem{AFW:2006:ECH}
{\sc D.~N. Arnold, R.~S. Falk, and R.~Winther}, {\em Finite element exterior
  calculus, homological techniques, and applications}, Acta Numer.,  (2006),
  pp.~1--155.


\bibitem{ArnolFalkWinth07}
\leavevmode\vrule height 2pt depth -1.6pt width 23pt, {\em Mixed finite element
  methods for linear elasticity with weakly imposed symmetry}, Math. Comp., 76
  (2007), pp.~1699--1723 (electronic).

\bibitem{ArnolWinth02}
{\sc D.~N. Arnold and R.~Winther}, {\em Mixed finite elements for elasticity},
  Numer. Math., 92 (2002), pp.~401--419.

\bibitem{ArnoldWintherNC}
{\sc D.~N. Arnold and R.~Winther}, {\em Nonconforming mixed elements for elasticity}, 
Dedicated to Jim Douglas, Jr. on the occasion of his 75th birthday.  Math. Models 
Methods Appl. Sci.  13  (2003),  no. 3, 295--307.
  
\bibitem{BrezziFortin1991}
{\sc F. Brezzi, and M. Fortin}, {\em Mixed and hybrid finite element methods}, Springer-Verlag, ISBN-10: 3540975829 (1991).  
  
\bibitem{CarstDolzmFunke00}
{\sc C.~Carstensen, G.~Dolzmann, S.~A. Funken, and D.~S. Helm}, {\em
  Locking-free adaptive mixed finite element methods in linear elasticity},
  Comput. Methods Appl. Mech. Engrg., 190 (2000), pp.~1701--1718.

\bibitem{CockbGopalGuzma10}
{\sc B.~Cockburn, J.~Gopalakrishnan, and J.~Guzm{\'{a}}n}, {\em A new
  elasticity element made for enforcing weak stress symmetry}, Math. Comp., 79
  (2010), pp.~1331--1349.

\bibitem{hpbook}
{\sc L.~Demkowicz}.
\newblock {\em Computing with $hp$ Finite Elements. I. One- and Two-Dimensional
  Elliptic and Maxwell Problems}.
\newblock Chapman \& Hall/CRC Press, Taylor and Francis, October 2006.

\bibitem{Demko:2008:PBI}
{\sc L.~Demkowicz}, {\em Polynomial exact sequences and projection-based
  interpolation with application to maxwell equations}, in Lecture Notes in
  Mathematics, Springer-Verlag, 2008.

\bibitem{DemkoBuffa:2005:PBI}
{\sc L.~Demkowicz and A.~Buffa}, {\em {$H\sp 1$}, {$H({\rm curl})$} and
  {$H({\rm div})$}-conforming projection-based interpolation in three
  dimensions. {Q}uasi-optimal {$p$}-interpolation estimates}, Comput. Methods
  Appl. Mech. Engrg., 194 (2005), pp.~267--296.

\bibitem{DemkoGopal:DPGanl}
{\sc L.~Demkowicz and J.~Gopalakrishnan}, 
\newblock Analysis of the {DPG} method for the {P}oisson equation.
\newblock {\em SIAM J Numer. Anal.}, 49(5):1788--1809, 2011.


\bibitem{DemkoGopal:2010:DPG1}
\leavevmode\vrule height 2pt depth -1.6pt width 23pt, {\em A class of
  discontinuous {P}etrov-{G}alerkin methods. {Part I:} {T}he transport
  equation}, Computer Methods in Applied Mechanics and Engineering, 199 (2010),
  pp.~1558--1572.

\bibitem{DemkoGopal:2010:DPG2}
\leavevmode\vrule height 2pt depth -1.6pt width 23pt, {\em A class of
  discontinuous {P}etrov-{G}alerkin methods. {P}art {II}: Optimal test
  functions}, Numerical Methods for Partial Differential Equations, 27 (2011),
  pp.~70--105.

\bibitem{DemkoGopalNiemi:2010:DPG3}
{\sc L.~Demkowicz, J.~Gopalakrishnan, and A.~Niemi}, 
\newblock A class of discontinuous {P}etrov-{G}alerkin methods. {P}art {III}:
  Adaptivity.
\newblock {\em Applied Numerical Mathematics}, 62:396--427, 2012.

\bibitem{DemkoGopalSchob:2008:EOP1}
{\sc L.~Demkowicz, J.~Gopalakrishnan, and J.~Sch{\"{o}}berl}, {\em Polynomial
  extension operators. {P}art~{I}}, SIAM J. Numer. Anal., 46 (2008),
  pp.~3006--3031.

\bibitem{DemkoGopalSchob:2009:EOP2}
\leavevmode\vrule height 2pt depth -1.6pt width 23pt, {\em Polynomial extension
  operators. {P}art~{II}}, SIAM J. Numer. Anal., 47 (2009), pp.~3293--3324.

\bibitem{DemkoGopalSchob:2009:EOP3}
\leavevmode\vrule height 2pt depth -1.6pt width 23pt,
\newblock Polynomial extension operators. {Part~III}.
\newblock {\em Math. Comp.}, 81(279):1289--1326, 2012.

\bibitem{Falk08}
{\sc R.~S. Falk}, {\em Finite element
  methods for linear elasticity}, in Mixed Finite Elements, Compatibility
  Conditions, and Applications, pp.~160--194.
\newblock Lectures given at the C.I.M.E. Summer School held in Cetraro, Italy,
  June 26-July 1, 2006,.

\bibitem{FungTong01}
{\sc Y.~C. Fung and P.~Tong}.
\newblock {Classical and Computational Solid Mechanics}.
\newblock World Scientific Publishing Co., 2001.

\bibitem{GopalGuzma10b}
{\sc J.~Gopalakrishnan and J.~Guzm{\'a}n}, 
\newblock A second elasticity element using the matrix bubble.
\newblock {\em IMA J. Numer. Anal.}, 32:352--372, 2012.

\bibitem{GopalGuzma10a}
{\sc J.~Gopalakrishnan and J.~Guzm{\'{a}}n}, 
\newblock Symmetric non-conforming mixed finite elements for linear elasticity.
\newblock {\em SIAM J Numer. Anal.}, 49(4):1504--1520, 2011.

\bibitem{Grisv92}
{\sc P.~Grisvard}, {\em Singularities in boundary value problems}, vol.~22 of
  Recherches en Math\'ematiques Appliqu\'ees [Research in Applied Mathematics],
  Masson, Paris, 1992.

\bibitem{Hager89}
{\sc W.~W. Hager}, {\em Updating the inverse of a matrix}, SIAM Rev., 31
  (1989), pp.~221--239.

\bibitem{Horgan:1995:KornIneq}
{\sc C.~O. Horgan}, {\em Korn’s inequalities and their applications in
  continuum mechanics}, SIAM Review, 37 (1995), pp.~491--511.

\bibitem{ManHuShi09}
{\sc H.-Y. Man, J. Hu and Z.-C. Shi}, {\em Lower order rectangular nonconforming mixed finite element for 
the three-dimensional elasticity problem},  Math. Models Methods Appl. Sci.  19  (2009),  no. 1, 51--65. 

\bibitem{MH:1993:MFE}
{\sc J.~E. Marsden and T.~J. Hughes}, {\em Mathematical foundations of
  elasticity}, Dover, 1993.
  
\bibitem{QiuDemko09}
{\sc W.~Qiu and L.~Demkowicz}, {\em Mixed $hp$-finite element method for linear
  elasticity with weakly imposed symmetry}, Computer Methods in Applied Mechanics and Engineering, 198 (2009),
  pp.~3682--3701.  

\bibitem{QiuDemko11}
{\sc W.~Qiu and L.~Demkowicz}, {\em Mixed $hp$-finite element method for linear
  elasticity with weakly imposed symmetry: stability analysis}, SIAM J. Numer.
  Anal., 49 (2011), pp.~619--641.


\bibitem{RaviaThoma77}
{\sc P.-A. Raviart and J.~M. Thomas}, {\em A mixed finite element method for
  2nd order elliptic problems}, in Mathematical aspects of finite element
  methods (Proc. Conf., Consiglio Naz. delle Ricerche (C.N.R.), Rome, 1975),
  Springer, Berlin, 1977, pp.~292--315. Lecture Notes in Math., Vol. 606.

\bibitem{Stenb88}
{\sc R.~Stenberg}, {\em A family of mixed finite elements for the elasticity problem}, Numer. Math., 53 (1988),
  pp.~513--538.

\bibitem{Vasil88}
{\sc D.~Vasilopoulos}, {\em On the determination of higher order terms of
  singular elastic stress fields near corners}, Numer. Math., 53 (1988),
  pp.~51--95.

\bibitem{ZitelMugaDemko:2010:DPG4}
{\sc J.~Zitelli, I.~Muga, L.~Demkowicz, J.~Gopalakrishnan, D.~Pardo, and
  V.~Calo}, {\em A class of discontinuous {P}etrov-{G}alerkin methods. {P}art
  {IV}: {W}ave propagation}, Journal of Computational Physics, 230 (2011),
  pp.~2406--2432.
\end{thebibliography}


\end{document}